\documentclass[12pt,twoside, final]{article}
\usepackage[utf8]{inputenc}
\usepackage{amsmath,amsthm,amssymb}
\usepackage{mathtools, tikz-cd}
\usepackage{color}
\usepackage{tabularx}
\usepackage{lmodern}
\usepackage{ragged2e}
\usepackage{hyperref}
\usepackage{bbold}

\theoremstyle{definition}
\newtheorem{definition}{Definition}[section]

\theoremstyle{plain}
\newtheorem{proposition}{Proposition}[section]
\newtheorem{theorem}{Theorem}[section]

\theoremstyle{remark}

\theoremstyle{definition}
\newtheorem{remark}{Remark}[section]

\newtheorem{example}{Example}[section]

\DeclareMathOperator{\del}{\partial}

\DeclareMathOperator{\D}{\operatorname{d}}

\DeclareMathOperator{\Pf}{\operatorname{Pf}}

\DeclareSymbolFont{bbold}{U}{bbold}{m}{n}
\DeclareSymbolFontAlphabet{\mathbbold}{bbold}

\let\sgn\relax\DeclareMathOperator{\sgn}{\operatorname{sgn}}

\newcommand{\iprod}{\mathbin{\lrcorner}}

\usepackage{pb-diagram}
\usepackage[matrix,arrow,curve]{xy}

\title{Generalized Geometries Constructed from Differential Forms on Cotangent Bundle}
\author{Radek Suchánek\footnote{Department of Mathematics and Statistics, Masaryk University, Brno, Czech Republic}} 

\date{\today}

\begin{document}

\maketitle





\noindent \textsc{Abstract.} We investigate the landscape of generalized geometries that can be derived from Monge-Ampère structures. Instead of following the approaches of Banos, Roubtsov, Kosmann-Schwarzbach, and others, we take a new path, inspired by the results of Hu, Moraru, and Svoboda. We construct a large family of new generalized almost geometries derived from non-degenerate 2D symplectic Monge-Ampère structures and other related geometric objects, such as complex structures. We demonstrate that, under certain assumptions, non-degenerate Monge-Ampère structures give rise to quadric surfaces of generalized almost geometries. Additionally, we discuss the link between Monge-Ampère structures and Monge-Ampère equations within this framework.

\tableofcontents

\section{Introduction}

Construction of generalized structures directly from tensor fields of the corresponding block representation was considered by many authors \cite{ BANOS2007841, Crainic2004GeneralizedCS, Gualtierri2011,hu2019commuting, Kosmann-Schwarzbach2010, Salvai2015, VAISMAN201584}. An example of this approach relevant to our case was Crainic's paper \cite{Crainic2004GeneralizedCS}, where the relations for tensor fields defining generic (integrable) generalized complex structures were described. Using the notion of a twist of a $2$-form by an endomorphism, Crainic proved the correspondence between generalized complex structures and Hitchin pairs. By Lychagin's results \cite{kushner_lychagin_rubtsov_2006, Lychagin, Lychagin-Diff_eq_on_2D_mflds}, every 2D Monge-Ampère equation can be encoded by a pair of $2$-forms, which moreover determine a specific endomorphism. Applying Monge-Ampère theory in the context of generalized complex geometry, Banos showed in \cite{BANOS2007841} that Monge-Ampère structures of divergence type give rise to generalized almost complex structures, and integrability of these structures is connected with the local equivalence problem for 2D symplectic Monge-Ampère equations with non-vanishing Pfaffian (i.e. elliptic or hyperbolic equations). This was further used by Kosmann-Schwarzbach and Rubtsov in \cite{Kosmann-Schwarzbach2010} to deform a Lie algebroid structure on $\mathbb{T} \left( T^*\mathcal{B} \right)$, which further induces a Courant algebroid structure on $\mathbb{T} \left( T^*\mathcal{B} \right)$ called by the authors \textit{Monge-Ampère Courant algebroid}. 

In this paper, we explore further the landscape of generalized geometries, which can be derived from Monge-Ampère structures. There were two apparent possibilities to do that. Either go into higher dimensions, especially dimension three, where symplectic classification is still possible or use a completely different method than Banos did. We took the latter path. Motivated by the results in \cite{hu2019commuting}, where anticommutative pairs and generalized metric compatible structures are considered, we constructed many new generalized almost geometries derived from non-degenerate 2D symplectic Monge-Ampère structures and from other geometric objects, they define (e.g. a complex structure and a pseudo-Riemannian metric). Some of these generalized almost geometries can be proven to be integrable by virtue of results in \cite{Banos2006IntegrableGA, BANOS2007841, Kosmann-Schwarzbach2010, kushner_lychagin_rubtsov_2006, Lychagin-Rubtsov-1983}. On the other hand, the standard notion of integrability via Nijenhuis tensor (or equivalently via Courant involutivity of certain almost Dirac structures) cannot be applied to many of the geometries constructed by our method. For example, for non-isotropic cases, one can consider the notion of \textit{weak integrability} instead \cite{hu2019commuting, SvobodaFreidel2019}. 
\hfill \\
\hfill \\

\noindent \textbf{Original content of the paper.} With the help of Monge-Ampère theory \cite{Lychagin, Lychagin-Diff_eq_on_2D_mflds, kushner_lychagin_rubtsov_2006}, we construct a family of generalized almost structures associated to 2D symplectic M-A structures. We show that, under certain assumptions, non-degenerate M-A structures give rise to quadric surfaces of generalized almost geometries. In this framework, we discuss the link between M-A structures and M-A equations.

\section{Monge-Ampėre structures and equations}

Let $\mathcal{M}$ be a $n$-dimensional symplectic manifold. In this case, $n$ must be an even number. We now define a specific pair of differential forms, which gives rise to other geometric structures that we will study in the next sections. 


\begin{definition}\label{def: M-A structure}
A pair $(\Omega, \alpha) \in \Omega^2 \left( \mathcal{M} \right) \times \Omega^n \left( \mathcal{M} \right)$ is called \textit{a Monge-Ampère structure over $\mathcal{M}$} (or just \textit{M-A structure}), if $\Omega$ is a symplectic form and 
    \begin{align*}
        \Omega \wedge \alpha = 0 \ . 
    \end{align*}
If $\mathcal{M} = T^* \mathcal{B}$, then $(\Omega, \alpha)$ is called \textit{a symplectic Monge-Ampère structure}. 
\end{definition}


\noindent \textit{Notation.} In our considerations, we will be dealing with symplectic M-A structures, and the symbol $\Omega$ will always denote the canonical symplectic form of the cotangent bundle.


\subsection{Monge-Ampėre theory in 2D}\label{subsection: 2D Monge-Ampėre theory}

We now consider symplectic 2D M-A structures and show how they look in canonical symplectic coordinates. We then proceed with the description of the corresponding M-A operators and equations. We will discuss some unique features of 2D M-A theory and describe various tensor fields that can be constructed from (non-degenerate) 2D M-A structures, and which are crucial for our further considerations.

\hfill 

Let $\dim \mathcal{B} = 2$, and recall that $\Omega$ denotes the canonical symplectic form on the cotangent bundle $T^* \mathcal{B}$. In the Darboux coordinates $x,y,p,q$ ($x,y$ are the base coordinates on $\mathcal{B}$), the symplectic form writes as
    \begin{align}\label{eq: symplectic form}
        \Omega = \D x \wedge \D p + \D y \wedge \D q \ .
    \end{align}
Let $\alpha$ be a $2$-form given by 
    \begin{align}\label{def: general alpha for 2D SMAE}
        \begin{split}
        \alpha & = A \D p \wedge \D y + B( \D x \wedge \D p - \D y \wedge \D q) \\ 
        & \qquad + C \D x \wedge \D q + D \D p \wedge \D q + E \D x \wedge \D y \ , 
        \end{split}
    \end{align} 
where the coefficients $A,B,C,D,E$ are smooth functions on $T^*\mathcal{B}$. Then $\alpha \wedge \Omega = 0$ holds for all possible choice of the coefficients, and $(\Omega, \alpha)$ is a M-A structure. 


Now let $f \colon \mathcal{B} \to \mathbb{R}$ be a smooth function. The differential of $f$ determines a section, $\D f \colon \mathcal{B} \to T^*\mathcal{B}$, given by $\D f (x) : = \D_x f$, and we can pullback $\alpha$ onto $\mathcal{B}$ to obtain a top form $(\D f )^* \alpha \in \Omega^2(\mathcal{B})$. Then the equation
    \begin{align}\label{def: M-A equation}
        (\D f )^* \alpha = 0
    \end{align}
defines a nonlinear second-order PDE for $f$ in two variables. If $\alpha$ is given by \eqref{def: general alpha for 2D SMAE}, then the equation \eqref{def: M-A equation} corresponds to \textit{2D symplectic Monge-Ampère equation}
    \begin{align}\label{eq: M-A equation in coordinates}
        A f_{xx} + 2B f_{xy} + C f_{yy} + D \left( f_{xx} f_{yy} - {f_{xy}}^2 \right) + E  = 0  \ ,
    \end{align}
where $f_{xy} : = \frac{\partial^2 f}{\partial x \partial y}$, and $A,B,C,D,E$ are smooth functions, which depend on $x,y, f_x, f_y$.

In this way, we can represent M-A equations via a pair of differential forms, the M-A structures. Nonetheless, this representation has a certain ambiguity, which can be partially removed with the notion of \textit{effectivity}. The condition $\Omega \wedge \alpha = O$ can be viewed as a definition of $\alpha$ being effective.

For a detailed exposition of the theory of Monge-Ampère equations and their applications, particularly in dimensions two and three, see \cite{kushner_lychagin_rubtsov_2006}. For further details about applications of Monge-Ampère theory, for example in the theory of complex differential equations, see \cite{BANOS20112187}, theoretical meteorology, and incompressible fluid theory, see \cite{Rubtsov_1997, RoubRoul2001, roulstone2009kahler, phdthesis, VolRoul2015}.


\begin{example}\label{example: differential forms for the Laplace and von Karman eqs}
Let $\alpha = - \D x \wedge \D q + \D y \wedge \D p$. Then 
    \begin{align*}
        (\D f )^* \alpha =  - \D x \wedge \D(f_y) + \D y \wedge \D(f_x) = (-f_{yy} -f_{xx} ) \D x \wedge \D y \ . 
    \end{align*}
Hence $(\D f )^* \alpha = 0$ amounts to the 2D Laplace equation $f_{yy} + f_{xx} = 0$ and $(\Omega, \alpha)$ is the corresponding M-A structure. Now let $\alpha = p \D p \wedge dy + \D x \wedge \D q$. Then $(\D f )^* \alpha = 0$ describes the von Karman equation $f_x f_{xx} - f_{yy} = 0$. 
\end{example}


\begin{definition}[\cite{kushner_lychagin_rubtsov_2006}]\label{def: Pfaffian, elliptic/hyperbolic/normalized structure}
\textit{Pfaffian} of a 2D M-A structure, $\Pf (\alpha) $, is defined by 
    \begin{align}\label{def: Pfaffian}
        \alpha \wedge \alpha = \Pf (\alpha) \Omega \wedge \Omega \ .
    \end{align}
M-A structure is called \textit{non-degenerate} if the Pfaffian is nowhere-vanishing. If $\Pf (\alpha) > 0$, then she structure is called \textit{elliptic}, and if $\Pf (\alpha) < 0$, then it is called \textit{hyperbolic}. We call a non-degenerate structure \textit{normalized}, if $|\Pf (\alpha)| = 1$. 
\end{definition} 


\begin{remark}
In the modelling of stably stratified geophysical flows, the Pfaffian is related with the Rellich's parameter \cite{viudez_dritschel_2002, dritschel_viudez_2003, Rubtsov2019}.
\end{remark}


For a general 2D symplectic M-A structure, the Pfaffian in canonical coordinates writes as
    \begin{align}\label{Pfaffian in canonical coordinates}
        \Pf (\alpha) = -B^2 + AC - DE \ .
    \end{align}
This can be directly checked using coordinate expressions \eqref{eq: symplectic form} and \eqref{def: general alpha for 2D SMAE}.


\begin{example}\label{example: Pfaffian for von Karma}
Let $(\Omega, \alpha)$ be the M-A structure described in example \ref{example: differential forms for the Laplace and von Karman eqs}, i.e. structure for 2D Laplace equation. Then $\operatorname{Pf} (\alpha) = 1$. Now consider the M-A structure for the von Karman equation $f_x f_{xx} - f_{yy} = 0$. Then $\operatorname{Pf} (\alpha) = p$. By Lychagin-Rubtsov theorem \ref{Lychagin-Rubtsov theorem} (see section \ref{subsection: endomorphisms and symmetric form}), the 2D Laplace equation gives rise to integrable complex structure, while the von Karman equation does not. 
\end{example}


\noindent \textbf{Normalization.} A non-degenerate M-A structure can be normalized $(\Omega, \alpha) \mapsto (\Omega, n(\alpha) )$ by 
    \begin{align}\label{def: normalization}
        n (\alpha) : = { |\Pf (\alpha)| }^{ -\frac{1}{2} } \alpha
    \end{align}
Indeed, the Pfaffian of $n (\alpha)$ satisfies $|\Pf \left( n(\alpha) \right)| = 1$, since
    \begin{align*}
        n (\alpha) \wedge n (\alpha) =  { |\Pf (\alpha)| }^{-1} \alpha \wedge \alpha = \frac{ \Pf (\alpha) }{ |\Pf (\alpha)| } \Omega \wedge \Omega \ ,
    \end{align*}
which implies
    \begin{align}\label{eq: Pfaffian of a normalized structure}
        \Pf \left( n(\alpha) \right) = \sgn \Pf (\alpha) \ .
    \end{align}
A non-degenerate M-A structure $(\Omega, \alpha)$ and its normalization $(\Omega, n(\alpha) )$ correspond to the same M-A equation, since 
    \begin{align*}
         (\D f)^* n(\alpha) = \left. { |\Pf (\alpha)| }^{ -\frac{1}{2} }\right\vert_{\operatorname{Im} \D f} (\D f)^* \alpha = 0 \ . 
    \end{align*}


\begin{remark}
Notice that we could have rescaled $\alpha$ with an arbitrary non-vanishing function, which would result in a new M-A structure $(\Omega, \Tilde{\alpha})$. Then, by the same argument as for the normalization, $(\Omega, \Tilde{\alpha})$ defines the same M-A equation as $(\Omega, \alpha)$. The main reason we choose to work with normalized structures is that generalized geometries that we will construct from $\alpha$ are not invariant with respect to the rescaling. Hence the normalization condition provides a consistent choice of the representative in the class $[\alpha]$, where $\tilde{\alpha} \in [\alpha]$, if and only if $\Tilde{\alpha} = e^{h}\alpha$ for some function $h$. Of course, the same argument also holds for other structures derived from $\alpha$.
\end{remark}

\subsection{Field of endomorphisms}\label{subsection: endomorphisms and symmetric form}

In the following paragraphs, we want to describe how every non-degenerate M-A structure defines a field of endomorphisms, which square to either $\operatorname{Id}_{T \left( T^*\mathcal{B} \right)}$, or $-\operatorname{Id}_{T \left( T^*\mathcal{B} \right)}$. Before we do so, we need to fix some notation.

\hfill 

\noindent \textit{Notation.} Tensor $\sigma \in \Gamma \left(T^*\mathcal{B} \otimes T^*\mathcal{B} \right)$ can be identified with a $C^\infty(\mathcal{B})$-linear map $\sigma_\# \colon \Gamma \left( T\mathcal{B} \right) \to \Gamma \left( T^*\mathcal{B} \right)$ defined by $ \sigma_\# (X) : = X \iprod \sigma$. Similarly, if $\tau \in \Gamma \left( T\mathcal{B} \otimes T\mathcal{B} \right)$, then we denote by $\tau^\#$ the linear map $\tau^\# \colon \Gamma \left(T^*\mathcal{B} \right) \to \Gamma \left( T\mathcal{B} \right)$, where $\tau^\# (\xi) : = \xi \iprod \tau$. Now consider a non-degenerate $2$-form $\alpha \in \Omega^2(T^*\mathcal{B})$, which means that $\alpha_\# \colon T\mathcal{B} \to T^*\mathcal{B}$ is an isomorphism. Using the inverse $(\alpha_\#)^{-1} \colon T^*\mathcal{B} \to T\mathcal{B}$, we define the bivector $\pi_\alpha \in \Gamma \left( \Lambda^2 T  T^*\mathcal{B} \right)$ by 
    \begin{align}\label{def: bivector corresponding to 2-form}
        (\pi_\alpha)^\# : = (\alpha_\#)^{-1} \ . 
    \end{align}
For example, if $\Omega = \D x\wedge \D p + \D y\wedge \D q$ is the canonical symplectic form, then $\pi_\Omega = \del_x \wedge \del_p + \del_y \wedge \del_q $ is the corresponding bivector, where $\partial_x : = \frac{\partial}{\partial_x}$ is the $x$-coordinate vector field (and similarly for the other fields $\partial_i$).

When working with matrices, we will use the underline notation to make the distinction between a morphism and its matrix representation. For example, if $\rho \in \operatorname{End} \left( T \left( T^*\mathcal{B} \right) \right)$, then corresponding matrix will be denoted $\underline{\rho}$ (which will always be understood with respect to the canonical coordinates of $\Omega$). Also, we will be omitting the $\#$ symbol when dealing with the morphisms derived from $2$-forms (and $2$-vectors), e.g. the matrix of $\alpha_\#$ will be denoted simmply $\underline{\alpha}$. Similarly, the matrix of $(\pi_\alpha)^\# $ will be denoted simply by $\underline{\alpha}^{-1}$ (see defining equation \eqref{def: bivector corresponding to 2-form}). When dealing with generalized structures, we write $\mathbb{J}$ for coordinate-free description, as well as for the coordinate description via matrices. Since block description of generalized structure either contains coordinate-free objects or their matrices, and the distinction between the two is clear in our notation, there is not much space for confusion.

\hfill 

\noindent \textbf{Almost complex and almost product structure.} Let $\rho \in \operatorname{End} \left( T T^*\mathcal{B} \right)$ be an endomorphism defined by
    \begin{align}\label{def: rho_alpha tensor}
        \rho : ={ | \Pf (\alpha) | }^{ - \frac{1}{2} } \pi_\Omega^\# \circ \alpha_\# \ .
    \end{align}
If $(\Omega, \alpha)$ is elliptic, then $\rho^2 = -\operatorname{Id}_{ T \left( T^*\mathcal{B} \right) }$, if $(\Omega, \alpha)$ is hyperbolic, then $\rho^2 = \operatorname{Id}_{ T \left( T^*\mathcal{B} \right) } $. The fact that $\rho$ is either an almost complex structure on $T^*\mathcal{B}$ (if $\Pf (\alpha) > 0$), or an almost product structure (if $\Pf (\alpha) < 0$), was proven in \cite{Lychagin-Rubtsov-1983}. In the canonical coordinates,
    \begin{align}\label{rho structure matrix - general case}
        \underline{\rho} = { | \Pf (\alpha) | }^{- \frac{1}{2} }
            \begin{pmatrix}
            B & -A & 0 & -D \\
            C & -B & D & 0 \\
            0 & E & B & C \\
            -E & 0 & -A & -B 
            \end{pmatrix} 
    \end{align}
and, as expected, it follows that 
    \begin{align*}
        \underline{\rho}^2 = \frac{-\Pf (\alpha)}{ | \Pf (\alpha) | } \operatorname{Id}_{ T \left( T^*\mathcal{B} \right) } = -\sgn \Pf (\alpha) \operatorname{Id}_{ T \left( T^*\mathcal{B} \right) } \ . 
    \end{align*}
Notice that $\rho$ is invariant of the normalization, since \eqref{eq: Pfaffian of a normalized structure} yields
    \begin{align*}
        \sgn \Pf \left( n(\alpha) \right) = \sgn^2 \Pf (\alpha) = \sgn \Pf (\alpha) \ . 
    \end{align*}

\hfill 

\noindent \textbf{Integrability.} V. Lychagin and V. Rubtsov showed in \cite{Lychagin-Rubtsov-1983} that there is a~direct link between the local equivalence of M-A equations and integrability of the $\rho$ structure derived from the corresponding M-A structure $(\Omega, \alpha)$. Moreover, the integrability condition can be expressed as a certain closedness condition. 

\begin{proposition}[Lychagin-Rubtsov \cite{Lychagin-Rubtsov-1983}]\label{Lychagin-Rubtsov theorem}
A 2D symplectic M-A equation $(\D f)^* \alpha = 0$ can be locally transformed via a symplectic transformation to either the Laplace equation $\Delta f = 0$, or the wave equation $\Box f = 0$, if and only if the $\rho$ structure is integrable, which is equivalent to $ \frac{ \alpha }{ \sqrt{ | \Pf (\alpha) | } }$ being closed. 
\end{proposition}


\section{Generalized geometry of M-A structures}


We recall some key definitions and objects from generalized geometry. To avoid any confusion, we formulate these notions for a general smooth manifold $\mathcal{M}$ and then choose  $\mathcal{M} = T^*\mathcal{B}$ to investigate the generalized geometry of symplectic 2D M-A structures and 2D incompressible fluid flows.

\subsection{Generalized (almost) structures} 

A \textit{generalized tangent bundle} over $\mathcal{M}$ is the vector bundle 
    \begin{equation}\label{generalized tangent bundle}
        \begin{tikzcd}
             \mathbb{T}\mathcal{M} : = T\mathcal{M} \oplus T^*\mathcal{M} \arrow{r}{\pi} &  \mathcal{M} \ ,  
        \end{tikzcd}
    \end{equation} 
with the bundle projection $\pi$ defined by the the composition 
    \begin{equation*}
        \begin{tikzcd}
             \mathbb{T}\mathcal{M} \arrow{r}{} &  T\mathcal{M} \arrow{r}{} & \mathcal{M} \ ,  
        \end{tikzcd}
    \end{equation*} 
where the left map is the projection on the first factor of the Whitney sum, and the second map is the tangent bundle projection on $\mathcal{M}$. The pairing between vector fields and $1$-forms endows $\mathbb{T}\mathcal{M}$ with a non-degenerate, symmetric, bilinear form $\eta$
    \begin{align}\label{def: inner product on the generalized tangent bundle}
        \eta \left( \left( X, \xi \right),  \left( Y, \zeta \right) \right) : = \frac{1}{2} \left( \xi (Y) + \zeta (X) \right) \ ,
    \end{align}
which defines on $\mathbb{T}\mathcal{M}$ a pseudo-Riemannian metric of signature $(n,n)$. If we choose a coordinate system $(q^\mu)$ on $\mathcal{M}$, then the corresponding coordinate vector fields and $1$-forms define a local basis $\left( (\partial_{q^\mu}, 0 ) , ( 0, dq^\mu ) \right) $ of $\mathbb{T} \mathcal{M}$. The matrix representation of $\eta$ in this basis is 
    \begin{align*}
        \eta = 
            \begin{pmatrix}
            0 & \mathbb{1} \\
            \mathbb{1} & 0
            \end{pmatrix} \ .
    \end{align*}
    

\textit{A generalized almost complex structure on $\mathcal{M}$} is a bundle map $\mathbb{J} \colon \mathbb{T}\mathcal{M} \to \mathbb{T}\mathcal{M}$ such that $\mathbb{J}^2 = - \operatorname{Id}_{ \mathbb{T}\mathcal{M}}$ and for all $ \left( X, \xi \right),  \left( Y, \zeta \right) \in \mathbb{T}\mathcal{M}$
    \begin{align}\label{eq: J acting on eta}
         \eta \left( \mathbb{J} \left( X, \xi \right),  \mathbb{J} \left( Y, \zeta \right) \right) = \eta \left( \left( X, \xi \right), \left( Y, \zeta \right) \right) \ , 
    \end{align}
where $\eta$ is the natural inner product \eqref{def: inner product on the generalized tangent bundle} \cite{Gualtierri2011, Hitchin-gen-calabi-yau}.


\begin{example}\label{example: Hitchin pairs and generalized complex structure}
Let $(\Omega, \alpha)$ be a non-degenerate M-A structure with $\alpha$ closed (i.e. a pair of symplectic structures). Consider a $(1,1)$-tensor $A_\alpha : =  \pi_\Omega^\# \circ \alpha_\#$ (c.f. definition \eqref{def: rho_alpha tensor}). Then $(\Omega, A_\alpha)$ defines \textit{a Hitchin pair of $2$-forms} in the sense of Crainic \cite{Crainic2004GeneralizedCS, Kosmann-Schwarzbach2010} since 
    \begin{align*}
        \Omega_\# A_\alpha = A_\alpha^* \Omega_\# \ .  
    \end{align*}
Using the notion of Hitchin pairs, B. Banos showed in \cite{BANOS2007841} that every 2D non-degenerate M-A structure satisfying for appropriate $\phi \in C^\infty \left( T^* \mathcal{B} \right) $ the \textit{divergence condition}
    \begin{align*}
        \D (\alpha + \phi \Omega ) = 0 \ , 
    \end{align*}
yields an integrable generalized almost structure $\mathbb{J}_\alpha$ given as follows
    \begin{align*}
        \mathbb{J}_\alpha = 
            \begin{pmatrix}
            A_\alpha & \pi_\Omega^\# \\
            - \left( \Omega_\# + \Omega_\# A_\alpha^2 \right) & -A_\alpha^*
            \end{pmatrix}
    \end{align*}
\end{example}


The result described in example \ref{example: Hitchin pairs and generalized complex structure} was a key motivation for our further investigations of other possibilities of constructing generalized (almost) geometries from Monge-Ampère equations and the corresponding Monge-Ampère structures. In order to state our result, we need to extend the notion of a~generalized (almost) complex structure by the following definition. 


\begin{definition}[\cite{hu2019commuting}] \label{def: generalized almost structure}
\textit{A generalized almost structure on $\mathcal{M}$} (or \textit{a generalized almost geometry on $\mathcal{M}$}) is a bundle map $\mathbb{J} \colon \mathbb{T}\mathcal{M} \to \mathbb{T}\mathcal{M}$ such that
    \begin{align*}
        \mathbb{J}^2 =  \gamma_1 \operatorname{id}_{\mathbb{T}M} \ , && \mathbb{J}^\bullet \eta = \gamma_2 \eta \ ,
    \end{align*}
where $\gamma_1, \gamma_2 \in \{ -1, 1\}$ and $\mathbb{J}^\bullet\eta \left( \left( X, \xi \right),  \left( Y, \zeta \right) \right) : = \eta \bigl( \mathbb{J} \left( X, \xi \right),  \mathbb{J} \left( Y, \zeta \right) \bigr)$. Table \ref{tab: type of generalized structure - general case} describes the four possible choices of constants $\gamma_i$.
    \begin{table}[h!]
        \begin{center}
        \begin{tabular}{ | >{\centering\arraybackslash}X p{2cm} || >{\centering\arraybackslash}X p{2cm}| >{\centering\arraybackslash}X p{2cm} |  >{\centering\arraybackslash}X p{2cm} |  >{\centering\arraybackslash}X p{2cm} |} 
            \hline 
            $(\gamma_1, \gamma_2)$ & $(1, 1)$ & $(1, -1)$ & $(-1, 1)$ & $(-1, -1)$ \\ 
            \hline
            type of $\mathbb{J}$  &  GaP & GaPC & GaC & GaAC  \\
            \hline
        \end{tabular}
        \caption{\label{tab: type of generalized structure - general case} Type of a generalized almost structure depending on $(\gamma_1, \gamma_2)$.}
        \end{center}
    \end{table}
The abbreviations stand for \textit{generalized almost product} (GaP), \textit{generalized almost complex} (GaC), \textit{generalized almost para-complex} (GaPC), and \textit{generalized almost anti-complex} (GaAC) structure. A generalized structure is called \textit{non-degenerate}, if its eigenbundles are isomorphic to $T \mathcal{M}$ (if $\mathbb{J}^2 = \operatorname{Id}_{ \mathbb{T} \mathcal{M} }$), or to $T \mathcal{M} \otimes \mathbb{C}$ (if $\mathbb{J}^2 = - \operatorname{Id}_{ \mathbb{T} \mathcal{M} }$)
\end{definition}


\begin{proposition}\label{lemma: diagonal antidiagonal structures}
Let $J, P \in \operatorname{End}(T\mathcal{M})$ be an almost complex and almost product structures, respectively. Let $\alpha \in \Omega^2 (\mathcal{M})$ be a non-degenerate $2$-form, and $g \in S^2(\mathcal{M})$ a non-degenerate symmetric bilinear form on $\mathcal{M}$. Suppose that $\mathbb{J}_J, \mathbb{J}_P, \mathbb{J}_\alpha, \mathbb{J}_g \in \operatorname{End}(\mathbbold{T}\mathcal{M})$ are given as follows 
    \begin{align}\label{diagonal and antidiagonal structures}
        \begin{split}
        \mathbb{J}_J  = 
            \begin{pmatrix}
            J & 0 \\
            0 & \epsilon J^*
            \end{pmatrix} \ , \qquad \qquad
        \mathbb{J}_\alpha = 
            \begin{pmatrix}
            0 & \pi_\alpha^\# \\
            \epsilon \alpha_\# & 0
            \end{pmatrix} \ , \\ 
        \mathbb{J}_P =      
            \begin{pmatrix}
            P & 0 \\
            0 & \epsilon P^*
            \end{pmatrix} \ ,  \qquad   \qquad
            \mathbb{J}_g =
            \begin{pmatrix}
            0 & \pi_g^\# \\
            \epsilon g_\# & 0
            \end{pmatrix}  \ ,
            \end{split}
    \end{align}
where $\epsilon \in \{-1, 1\}$. Then $\mathbb{J}_J, \mathbb{J}_P, \mathbb{J}_\alpha, \mathbb{J}_g$ are generalized almost structures and they satisfy
    \begin{align}\label{relations for diag. and antidiag structures}
        \begin{split}
         - {\mathbb{J}_J}^2 & = {\mathbb{J}_P}^2 = \operatorname{Id}_{\mathbb{T}M} \ , \ \qquad \qquad    - \mathbb{J}_J^\bullet \eta = \mathbb{J}_P^\bullet \eta = \epsilon \eta \ , \\ 
         {\mathbb{J}_\alpha}^2 & = {\mathbb{J}_g}^2 = \epsilon \operatorname{Id}_{\mathbb{T}M}  \ , \qquad \qquad - \mathbb{J}_\alpha^\bullet \eta =  \mathbb{J}_g^\bullet \eta = \epsilon \eta \ .
         \end{split}
    \end{align}
Types of these structures are given in the following table. 
    \begin{table}[h!]
        \begin{center}
        \begin{tabular}{ | >{\centering\arraybackslash}X p{2cm} | >{\centering\arraybackslash}X p{2cm}| >{\centering\arraybackslash}X p{2cm} |  >{\centering\arraybackslash}X p{2cm} |  >{\centering\arraybackslash}X p{2cm} |} 
            \hline 
            $\sgn{\epsilon}$    & $\mathbb{J}_J$ & $\mathbb{J}_P$ & $\mathbb{J}_\alpha $ & $\mathbb{J}_g$ \\ 
            \hline
            $+$  &  GaAC & GaP & GaPC & GaP  \\
            \hline
            $-$  & GaC & GaPC & GaC & GaAC \\
            \hline
        \end{tabular}
        \caption{\label{tab: type of generalized structure} Type of a generalized almost structure depending on $\sgn \epsilon$.}
        \end{center}
    \end{table}
\end{proposition}

\begin{proof}
Since $- J^2 =  P^2 = \operatorname{Id}_{T\mathcal{M}}$ and ${J^*}^2 = {J^2}^*$, it is easy to see that $ - {\mathbb{J}_J}^2 = {\mathbb{J}_P}^2 = \operatorname{Id}_{\mathbb{T}\mathcal{M}}$. Now we recall \eqref{def: bivector corresponding to 2-form}, which means 
    \begin{align}\label{relations 1}
        g_\#  \pi_g^\#  & = \operatorname{Id}_{T^*\mathcal{M}} = \alpha_\#  \pi_\alpha^\# \ , \\
        \pi_g^\#  g_\# & = \operatorname{Id}_{T\mathcal{M}} = \pi_\alpha^\#   \alpha_\# \ .
    \end{align}
This implies ${\mathbb{J}_\alpha}^2 = {\mathbb{J}_g}^2 = \epsilon \operatorname{Id}_{\mathbb{T}M}$. We proceed with the compatibility of $\mathbb{J}_J$ with $\eta$. From \eqref{def: inner product on the generalized tangent bundle} we have
    \begin{align*}
        \mathbb{J}_J^\bullet \eta \left( (X, \xi) , (Y, \zeta) \right) = \epsilon \left( J^* \xi (JY) + J^* \zeta (JX) \right) \ . 
    \end{align*}
Using the definition of the dual map, we arrive at
    \begin{align*}
        \mathbb{J}_J^\bullet \eta \left( (X, \xi) , (Y, \zeta) \right) = \epsilon \left( \xi (A^2 Y) + \zeta (A^2 X) \right) = -\epsilon \eta \left( (X, \xi) , (Y, \zeta) \right) \ .
    \end{align*}
The computation for $\mathbb{J}_P$ is completely analogous with the only difference coming from $P^2 = - J^2$, resulting in $\mathbb{J}_P^\bullet \eta = \epsilon \eta$. The computation for $\mathbb{J}_\alpha$ (and $\mathbb{J}_g$) is slightly different. We have
    \begin{align*}
        (\mathbb{J}_\alpha)^\bullet \eta \left( (X, \xi) , (Y, \zeta) \right) = \frac{\epsilon}{2} \left( \alpha_\# X (\pi_\alpha^\# \zeta) + \alpha_\# Y (\pi_\alpha^\# \xi) \right) \ . 
    \end{align*}
Since $\alpha_\# X = \alpha (X,-)$, and due to antisymmetry of $\alpha$,
    \begin{align*}      
            \frac{\epsilon}{2} \left( \alpha_\# X (\pi_\alpha^\# \zeta) + \alpha_\# Y (\pi_\alpha^\# \xi) \right) &  = \frac{\epsilon}{2} \left( \alpha(X, \pi_\alpha^\# \zeta) + \alpha( Y, \pi_\alpha^\# \xi) \right) \\ 
            & = \frac{-\epsilon}{2} \left( ( \alpha_\#  \pi^\# \zeta) X + (\alpha_\#  \pi_\alpha^\# \xi) Y \right) \ . 
    \end{align*}
Using the relations \eqref{relations 1} once again, we obtain
    \begin{align*}
         (\mathbb{J}_\alpha)^\bullet \eta \left( (X, \xi) , (Y, \zeta) \right) = -\epsilon \eta \left( (X, \xi) , (Y, \zeta) \right) \ .
    \end{align*}
The computation for $\mathbb{J}_g$ differs from the case of $\mathbb{J}_\alpha$ only in the symmetry of $g$. Thus $- \mathbb{J}_J^\bullet \eta = \mathbb{J}_P^\bullet \eta = \epsilon \eta $, and $- \mathbb{J}_\alpha^\bullet \eta =  \mathbb{J}_g^\bullet \eta = \epsilon \eta$. The table \ref{tab: type of generalized structure} summarizes the resulting type of structures according to definition \ref{def: generalized almost structure}.
\end{proof}


\begin{example}\label{example: non-isotropic structures are not totally isotropic}
Let $g$ be a non-degenerate symmetric bilinear form. Consider
    $ \mathbb{A} = 
    \begin{pmatrix}
    0 & \pi_g^\# \\
    g_\# & 0 \\
    \end{pmatrix} .
    $
w $\mathbb{A}$ satisfies $\mathbb{A} = - \operatorname{Id}_{\mathbb{T}\mathcal{M}}$ and $\mathbb{A}^\bullet \eta = \eta$, so $\mathbb{A}$ is a generalized almost product structure. The $\pm1$-eigenbundles of $\mathbb{A}$ are
    \begin{align*}
        E_\pm = \{ (X, \pm g_\# X ) | X \in \Gamma(T\mathcal{M}) \} \ .
    \end{align*}
Indeed, for $x_+ = (X, g_\# X ) \in E_+$, and $y_- = (Y, -g_\# Y ) \in E_-$, it holds $\mathbb{A} x_+ = (X, g_\# X ) = x_+$, and $\mathbb{A} y = (-Y, g_\# Y ) = -y_-$. Now suppose $x_+,y_+ \in E_+$. Then 
    \begin{align*}
        \eta(x_+,y_+) = \frac{1}{2}\left( (g_\# X) Y + (g_\# Y) X \right) = g(X,Y) \ .
    \end{align*}
Obviously there are $X,Y \in \Gamma(T\mathcal{M})$ such that $g(X,Y) \neq 0$, thus $E_+$ is not totally isotropic, and hence $\mathbb{A}$ is a non-isotropic structure. 
\end{example}


\hfill

\noindent \textbf{Coordinate description.} Since $\left( \Omega_\# (\partial_i) \right) (\partial_j) = \Omega (\partial_i, \partial_j)$, we have
    \begin{align*}
        \underline{\Omega} = \begin{pmatrix} 0 & \mathbb{1} \\ - \mathbb{1} & 0 \end{pmatrix} \ .
    \end{align*}
Similarly for the matrix of $\pi_\Omega^\#$ we have
    \begin{align*}
        \underline{\pi_\Omega}^\# = \underline{\Omega}^{-1} \ .
    \end{align*}
The same holds true for arbitrary non-degerate $2$-form $\alpha$. Hecne we will denote the matrices of $\alpha_\#$ and $\pi_\alpha^\#$ simply by $\underline{\alpha}$ and $\underline{\alpha}^{-1}$, respectively. The situation does not change when dealing with a symmetric non-degenerate $2$-tensor $g$. Finally, we recall that for an endomorphism $A$, the matrix of $A$ and the matrix of $A^*$ are related by the transpose $\underline{A}^* = \underline{A}^T$. Thus the generalized almost structures described in lemma \ref{lemma: diagonal antidiagonal structures}, which are the building blocks for our subsequent constructions, are written in the canonical coordinates as 
    \begin{align*}
        \mathbb{J}_J  = 
            \begin{pmatrix}
            J & 0 \\
            0 & \epsilon J^T
            \end{pmatrix} \ , \qquad \qquad
        \mathbb{J}_\alpha = 
            \begin{pmatrix}
            0 & \underline{\alpha}^{-1} \\
            \epsilon \underline{\alpha} & 0
            \end{pmatrix} \ , \\ 
        \mathbb{J}_P =      
            \begin{pmatrix}
            P & 0 \\
            0 & \epsilon P^T
            \end{pmatrix} \ ,  \qquad   \qquad
            \mathbb{J}_g =
            \begin{pmatrix}
            0 & \underline{g}^{-1}\\
            \epsilon \underline{g} & 0
            \end{pmatrix}  \ ,
    \end{align*}

\hfill

\noindent \textbf{The Courant bracket.} The space of sections $\Gamma \left(\mathbb{T}\mathcal{M} \right) = \Gamma \left( T\mathcal{M} \right) \oplus \Gamma \left( T^*\mathcal{M} \right)$ is equipped with the antisymmetric \textit{Courant bracket} $[ -,- ]_C$ \cite{Courant-DiracBracket}
    \begin{align}\label{def: Courant bracket}
       [ \left( X, \xi \right),  \left( Y, \zeta \right) ]_C : = \left( [X,Y] , \mathcal{L}_X \zeta -  \mathcal{L}_Y \xi - \frac{1}{2} \D ( X \iprod \zeta -  Y \iprod \xi ) \right) \ ,
    \end{align}
where $[X,Y]$ is the Lie bracket and $\mathcal{L}$ is the Lie derivative. Note that the Courant bracket, which is an extension of the Lie bracket for vector fields, does not satisfy the Jacobi identity. More importantly for our considerations, the Courant bracket can be used to define integrability of \textit{isotropic} structures on $\mathbb{T}\mathcal{M}$.

\subsubsection{Isotropy, Dirac structures, and integrability}

Every generalized almost structure comes together with the corresponding subbundles $E_{+} , E_{-}$, which are $\pm 1$-eigenbundles if $\mathbb{J}^2 = \operatorname{id}_{\mathbb{T}\mathcal{M} }$, and $\pm i$-eigenbundles if $\mathbb{J}^2 = -\operatorname{Id}_{\mathbb{T}\mathcal{M}}$. 

\begin{definition}
A subbundle $E \subset \mathbb{T}\mathcal{M}$ is called \textit{totally isotropic} (w.r.t. the inner product $\eta$), if for all $x,y \in E: \eta(x,y) = 0$. Totally isotropic $E$ is called an \textit{almost Dirac structure on $\mathcal{M}$} if $\operatorname{rank}E = \operatorname{rank}T\mathcal{M}$. A \textit{Dirac structure on $\mathcal{M}$} is an almost Dirac structure $E$ such that $[E, E]_C \subset E$.
\end{definition}

\begin{remark}
A Dirac structure on $\mathcal{M}$ can be equivalently defined as a totally isotropic subbundle $E \subset \mathbb{T}\mathcal{M}$ of maximal rank, which is involutive with respect to the Courant bracket.      
\end{remark}

The four possible generalized almost structures determined by definition \ref{def: generalized almost structure} can be divided into two subsets depending on whether the eigenbundles $E_\pm$ are almost Dirac structures or not. 

\begin{definition}
Let $\mathbb{J}$ be a generalized almost structure. If the eigenbundles $E_\pm$ are almost Dirac structures, then $\mathbb{J}$ is called \textit{isotropic} (with respect to $\eta$). Otherwise $\mathbb{J}$ is called \textit{non-isotropic}.
\end{definition}

The involutivity condition of almost Dirac structures $E_\pm$ can serve as a~definition of integrability only for isotropic $\mathbb{J}$ (see remark \ref{non-isotropy and integrability} below).

\begin{definition}\label{def: integrability of a generalized structure}
An isotropic generalized almost structure $\mathbb{J} \in \operatorname{End} (\mathbb{T} \mathcal{M}) $ is called \textit{integrable} if the corresponding eigenbundles $E_+, E_-$ are Dirac structures. An integrable generalized almost structure is called a \textit{generalized structure}. 
\end{definition}

Let $\mathbb{J} \in \operatorname{End} ( \mathbb{T}\mathcal{M}) $ be an isotropic generalized almost structure. Then the torsion of $\mathbb{J}$ is defined for all $x,y \in \Gamma ( \mathbb{T}\mathcal{M} )$ by 
    \begin{align}\label{def: torsion of A}
        N_{\mathbb{J}} (x,y) : = [\mathbb{J}x, \mathbb{J}y]_C + \mathbb{J}^2 [x,y]_C - \mathbb{J} \left( [ \mathbb{J}x, y ]_C + [x, \mathbb{J}y]_C \right) \ .
    \end{align}
This $(1,2)$-tensor $N_{\mathbb{J}} \colon \Gamma ( \mathbb{T}\mathcal{M} ) \otimes \Gamma ( \mathbb{T}\mathcal{M} ) \to \Gamma ( \mathbb{T}\mathcal{M} ) $ is called (generalized) \textit{Nijenhuis tensor}. An isotropic $\mathbb{A}$ is integrable if and only if the coresponding Nijenhuis tensor vanishes: $N_{\mathbb{A}} (x,y) = 0$ for all $x,y \in \Gamma (\mathbb{T}\mathcal{M})$ \cite{Crainic2004GeneralizedCS}. 

\begin{remark}\label{non-isotropy and integrability}
If a generalized almost structure $\mathbb{J}$ is non-isotropic, then the eigenbundles $E_\pm$ are not totally isotropic (see example \ref{example: non-isotropic structures are not totally isotropic}) and the Courant bracket is not well-defined on them. Thus the torsion \eqref{def: torsion of A} for non-isotropic structures is not well-defined as a tensor. As a consequence, the notion of integrability, which is usually given either by the condition of vanishing Nijenhuis tensor or by definition \ref{def: integrability of a generalized structure}, cannot be applied to non-isotropic structures. 
\end{remark}

\begin{remark}
Authors of \cite{hu2019commuting} considered the notion of \textit{weak integrablity}, which they defined for a commuting pair consisting of an indefinite generalized metric $\mathbb{G}$ and an arbitrary generalized structure $\mathbb{J}$. This notion of weak integrability can be applied to non-isotropic structures as well but requires the existence of a \textit{generalized Bismut connection}\footnote{A generalized Bismut connection associated to a generalized metric is a Courant algebroid connection that parallelizes the generalized metric \cite{Gualtieri2007BranesOP}. The extension for indefinite generalized metrics was given in \cite{hu2019commuting}.} $\mathfrak{D}$ associated to a generalized metric $\mathbb{G}$ \cite{Gualtieri2007BranesOP, hu2019commuting}. By definition, the weak integrability condition is $\mathfrak{D} \mathbb{J} = 0$ \cite{hu2019commuting}. In the case of generalized Kähler and generalized para-Kähler structures, the integrability condition introduced in \cite{Gualtieri2007BranesOP} for the case of generalized Kähler structures (and extended to the case of generalized para-Kähler structures in \cite{hu2019commuting}) implies weak integrability \cite{hu2019commuting}. 
\end{remark}


\subsection{Generalized geometry and M-A theory}


Recall that $\mathcal{B}$ is $2$-dimensional and $\Omega$ denotes the canonical symplectic form on $T^* \mathcal{B}$.

\begin{proposition}\label{cor: generalized geometry from generic M-A structures}
Let $(\Omega, \alpha) \in \Omega^2 \left( T^*\mathcal{B} \right) \times \Omega^2 \left( T^*\mathcal{B} \right)$ be a non-degenerate Monge-Ampère structure, $\rho \in \operatorname{End} \left( T T^*\mathcal{B} \right)$ the corresponding endomorphism defined by \eqref{def: rho_alpha tensor}. Then 
    \begin{align}\label{diag and antidiag structures for generic 2D M-A equation}
        \mathbb{J}_\rho = 
            \begin{pmatrix}
            \rho & 0 \\
            0 & \epsilon_1 \rho^* \ ,
            \end{pmatrix}   
            &&  
        \mathbb{J}_\alpha = 
            \begin{pmatrix}
            0 & \pi_\alpha^\# \\
            \epsilon_2 \alpha_\# & 0
            \end{pmatrix} \ , 
            &&  
        \mathbb{J}_\Omega = 
            \begin{pmatrix}
            0 & \pi_\Omega^\# \\
            \epsilon_3 \Omega_\# & 0
            \end{pmatrix} \ ,
    \end{align}
are generalized almost structures. 
    \begin{enumerate}
        \item If the M-A structure is elliptic, $\operatorname{Pf} (\alpha) > 0$, then $\mathbb{J}_\rho$ is a GaPC structure for $\epsilon_1 = 1$, and it is a GaC structure for $\epsilon_1 = -1$.
        \item If the M-A structure is hyperbolic, $\operatorname{Pf} (\alpha) < 0$, then $\mathbb{J}_\rho$ is a GaP structure for $\epsilon_1 = 1$, and it is a GaPC structure for $\epsilon_1 = -1$.
        \item The types of $\mathbb{J}_\alpha, \mathbb{J}_\Omega$ are independent of $\sgn \left(\operatorname{Pf} (\alpha) \right)$. The types are determined by the value of $\epsilon_1$, and $\epsilon_2$, respectively, and are summarized in the table \ref{tab: type of generalized structure}. The two structures never coincide, $\mathbb{J}_\alpha \neq \mathbb{J}_\Omega$.  
    \end{enumerate}
\end{proposition}

\begin{proof}
Firstly notice that the assumption of non-degeneracy of the M-A structure assures that $\mathbb{J}_\rho$ and $\mathbb{J}_\alpha$ can be considered. This can be seen from 
    \begin{align}\label{eq: M-A structure non-deg iff alpha non-deg.}
        \det \underline{\alpha} = {\Pf (\alpha)}^2 \ ,
    \end{align}
and recalling that non-degeneracy means that $\Pf (\alpha)$ is nowhere vanishing. 

For the sake of completeness, we proceed with showing that $\rho$ is either an almost complex structure if $\Pf (\alpha) > 0$, or an almost product structure if $\Pf (\alpha) < 0$. The equation \eqref{def: rho_alpha tensor} writes $ \underline{\alpha} = \sqrt{| \Pf (\alpha) | } \underline{\rho}^T \underline{\Omega}$, which implies $\underline{\rho} = \frac{1}{\sqrt{ | \Pf (\alpha) | }} (\underline{\alpha} \underline{\Omega}^{-1})^T$. Because $\underline{\Omega}, \underline{\alpha}$ are antisymmetric matrices, and $ \underline{\Omega}^{-1} = - \underline{\Omega}$, we obtain $ {\underline{\rho}}^2 = \frac{1}{ | \Pf (\alpha) | } ( \underline{\Omega} \underline{\alpha})^2$.
Since
    \begin{align}\label{matrices of alpha and Omega in canonical basis}
        \underline{\Omega} = 
        \begin{pmatrix}
        0 & 0 & 1 & 0 \\
        0 & 0 & 0 & 1 \\
        -1 & 0 & 0 & 0 \\
        0 & -1 & 0 & 0
        \end{pmatrix} \ ,
        &&  \underline{\alpha} = 
        \begin{pmatrix}
        0 & E & B & C \\
        -E & 0 & -A & -B \\
        -B & A & 0 & D \\
        -C & B & -D & 0 
        \end{pmatrix} \ ,
    \end{align}
we have $ ( \underline{\Omega} \underline{\alpha})^2 = (B^2 - AC + DE) \mathbb{1} $. Using the coordinate expression for Pfaffian, $\Pf (\alpha) = -B^2 + AC - DE $, we arrive at ${\underline{\rho}}^2 = - \sgn \Pf (\alpha) \mathbb{1}$. Thus ${\underline{\rho}}^2 = \mathbb{1} $ for $\Pf < 0$, which amounts to a hyperbolic M-A structures, and ${\underline{\rho}}^2 = - \mathbb{1} $ for $\Pf > 0$, which defines an elliptic M-A structure. These two cases are the only possible results for ${\underline{\rho}}^2$ due to the non-degeneracy of the M-A structure $(\Omega, \alpha)$. Now let $T \in \operatorname{Gl}(4, \mathbb{R})$ be a transformation (acting on a fiber of $TM$ above some point of $M$), which maps the basis $b$ induced by the canonical coordinates, to some other basis $\Tilde{b}$. Then the matrix of $\rho$ w.r.t. the $\Tilde{b}$ basis satisfies $\underline{\Tilde{\rho}} = T \underline{\rho} T^{-1} $, which implies ${\underline{\Tilde{\rho}}}^{\ 2} = T {\underline{\rho}}^2 T^{-1} = \pm \mathbb{1}$. Thus $\rho^2 = \pm \operatorname{Id}_{TM}$. 

That $\mathbb{J}_\rho$ is a GaP structure for $\epsilon_1 = 1$, and a GaPC structure for $\epsilon_1 = -1$, now follows from the proposition \ref{lemma: diagonal antidiagonal structures}. Similarly, the types of generalized almost structures $\mathbb{J}_\alpha$ and $\mathbb{J}_\Omega$ are given by the table \ref{tab: type of generalized structure}. The fact that this is independent of $\Pf (\alpha)$ follows from the proof of the aforementioned proposition, as well as from the assumption $\Pf (\alpha) \neq 0$, which assures that $\alpha$ is non-degenerate, and thus satisfies the assumptions of the proposition. Indeed, if $\Pf (\alpha) \neq 0$, then from \eqref{eq: M-A structure non-deg iff alpha non-deg.} follows that $\alpha_\#$ is invertible and hence $\alpha$ is non-degenerate.

Finally, the definition \ref{def: M-A structure} requires $\alpha \wedge \Omega = 0$. The non-degeneracy of the symplectic form $\Omega$ implies that $\alpha$ cannot be written as a sum of two forms $\alpha = \alpha_1 + \alpha_2$, such that one of the summands is colinear with $\Omega$. Thus the structures always satisfy $\mathbb{J}_\alpha \neq \mathbb{J}_\Omega$.
\end{proof}

\begin{remark}
A coordinate-free proof of $\rho$ being either an almost complex structure on $T^*\mathcal{B}$, if $\Pf (\alpha) > 0$, or an almost product structure if $\Pf (\alpha) < 0$, can be found in \cite{kushner_lychagin_rubtsov_2006}. 
\end{remark}


\subsection{Quadric surfaces of generalized geometries}


Consider a smooth manifold equipped with three almost complex structures $I,J,K$ satisfying $IJ + JI =0$ and $K = IJ$. The manifold is then called an almost hypercomplex manifold. On a hypercomplex manifold, there is a $2$-sphere of almost complex structures $\{a_1 I + a_2 J + a_3 K | \sum_{i = 1}^3 {a_i}^2 = 1 \}$. This follows from the fact that the above relations between $I,J,K$ imply anticommutativity of the triple $\{ I,J\} = \{ J,K \} = \{ I,K \} = 0 $. We will see that a similar but much richer situation happens when we consider a pair-wise anticommutative triple of generalized geometries constructed from tensors associated with 2D non-degenerate M-A structures. 

\begin{proposition}\label{prop.: pair-wise anticommutative diag/antidiag structures}
Let $(\Omega, \alpha) \in \Omega^2 \left( T^*\mathcal{B} \right) \times \Omega^2 \left( T^*\mathcal{B} \right)$ be a non-degenerate Monge-Ampère structure, $\rho \in \operatorname{End} (T T^*\mathcal{B})$ the corresponding endomorphism defined by \eqref{def: rho_alpha tensor}. Then the generalized almost structures $\mathbb{J}_\rho, \mathbb{J}_\alpha, \mathbb{J}_\Omega$ given by \eqref{diag and antidiag structures for generic 2D M-A equation} pair-wise anticommute, if and only if $\epsilon_1 = -1$ and 
    \begin{align}\label{condition for anticommutativity of J_alpha and J_Omega}
        \epsilon_3 \pi_\alpha^\# \Omega_\# + \epsilon_2 \pi_\Omega^\# \alpha_\# = 0 \ .
    \end{align}
In canonical coordinates, the condition \eqref{condition for anticommutativity of J_alpha and J_Omega} is equivalent with 
    \begin{align}\label{condition for anticommutativity of J_alpha and J_Omega in canonical coordinates}
        B^2 - AC + DE = - \epsilon_2 \epsilon_3 \ , 
    \end{align}
where $A,B,C,D,E \in C^\infty \left( T^*\mathcal{B} \right)$ are coefficients of $\alpha$ in the canonical basis.
\end{proposition}

\begin{remark}
In the proof of proposition \ref{cor: generalized geometry from generic M-A structures}, we have shown that $\Pf (\alpha) = -B^2 + AC - DE $. Thus the condition \eqref{condition for anticommutativity of J_alpha and J_Omega in canonical coordinates} can be expressed as $\Pf (\alpha) = \epsilon_2 \epsilon_3$.
\end{remark}

\begin{proof}
The anticommutators are 
    \begin{align*}
        \{ \mathbb{J}_\rho, \mathbb{J}_\alpha \} & =  
            \begin{pmatrix}
            0 & \rho \pi_\alpha^\# + \epsilon_1 \pi_\alpha^\# \rho^* \\
            \epsilon_1 \epsilon_2 \rho^* \alpha_\# + \epsilon_2 \alpha_\# \rho &  0 
            \end{pmatrix} \ , \\
        \{ \mathbb{J}_\rho, \mathbb{J}_\Omega \} &  =  
            \begin{pmatrix}
            0 & \rho \pi_\Omega^\# + \epsilon_1 \pi_\Omega^\#  \rho^* \\
            \epsilon_1 \epsilon_3 \rho^* \Omega_\# + \epsilon_3 \Omega_\# \rho &  0 
            \end{pmatrix} \ ,  \\
        \{ \mathbb{J}_\alpha, \mathbb{J}_\Omega \} &  =  
            \begin{pmatrix}
            \epsilon_3 \pi_\alpha^\# \Omega_\# + \epsilon_2 \pi_\Omega^\# \alpha_\# & 0 \\
            0 & \epsilon_2 \alpha_\# \pi_\Omega^\# + \epsilon_3 \Omega_\# \pi_\alpha^\#
            \end{pmatrix} \ .
    \end{align*}
The condition $\{ \mathbb{J}_\rho, \mathbb{J}_\alpha \} = 0 $ requires $\rho \pi_\alpha^\# + \epsilon_1 \pi_\alpha^\# \rho^* = 0 $. Multiplying this equation with $\alpha_\#$ from the right and left (i.e. pre-composition and composition), we obtain $\alpha_\# \rho + \epsilon_1 \rho^* \alpha_\# = 0$. The dual equation is $ \rho^* \alpha_\#^* + \epsilon_1 \alpha_\#^* \rho = 0$, which, by antisymmetry of $\alpha$, is equivalent with $ \epsilon_1 \epsilon_2 \rho^* \alpha_\# + \epsilon_2 \alpha_\# \rho  = 0$. Similarly   
    \begin{align*}
        \rho \pi_\Omega^\# + \epsilon_1 \pi_\Omega^\#  \rho^* = 0 \iff \epsilon_1 \rho^* \Omega_\# + \Omega_\# \rho = 0 \ . 
    \end{align*}
Thus the first two anticommutators vanish, if and only if 
    \begin{align}\label{intermediate computation}
        \begin{split}
        \epsilon_1 \rho^* \alpha_\# + \alpha_\# \rho & = 0 \ , \\
        \epsilon_1 \rho^* \Omega_\# + \Omega_\# \rho & = 0 \ .
        \end{split}
    \end{align}
Working with the canonical coordinates, the system \eqref{intermediate computation} writes
    \begin{align*}
        \frac{-1}{ \sqrt{| \Pf (\alpha) | } } \underline{\alpha}  \underline{\Omega} \left( \epsilon_1 \underline{\alpha} +  \underline{\alpha} \right) & = 0 \ , \\
        \frac{-1}{ \sqrt{| \Pf (\alpha) | } } {\underline{\Omega}}^2 \left( \epsilon_1 \underline{\alpha} +  \underline{\alpha} \right) & = 0 \ .
    \end{align*}
Due to non-degeneracy of the M-A structure $(\Omega, \alpha)$, the above equations hold, if and only if $\epsilon_1 = -1$. The remaining condition $\{ \mathbb{J}_\alpha, \mathbb{J}_\Omega \} = 0$ is captured by the equation \eqref{condition for anticommutativity of J_alpha and J_Omega} since
    \begin{align*}
        \epsilon_3 \pi_\alpha^\# \Omega_\# + \epsilon_2 \pi_\Omega^\# \alpha_\# = 0  \iff  \epsilon_3 \pi_\alpha^\# \Omega_\# + \epsilon_2 \pi_\Omega^\# \alpha_\# = 0 \ .
    \end{align*}
This follows directly from \eqref{def: bivector corresponding to 2-form}. In matrix notation, \eqref{condition for anticommutativity of J_alpha and J_Omega} is equivalent with
    \begin{align*}
       \left( \underline{\alpha} \underline{\Omega} \right)^2 = - \epsilon_2 \epsilon_3 \mathbb{1} \ . 
    \end{align*}
Recalling \eqref{matrices of alpha and Omega in canonical basis}, we can write $\underline{\alpha}$ and $\underline{\Omega}$ in the block form 
    \begin{align}\label{the condition on anticommutativity of J_alpha J_Omega in matrices}
        \underline{\alpha} = 
            \begin{pmatrix}
            P & Q \\
            -Q^T & R
            \end{pmatrix} && 
        \underline{\Omega} =     
            \begin{pmatrix}
            0 & \mathbb{1} \\
            -\mathbb{1} & 0
            \end{pmatrix} \ ,
    \end{align}
where $P^T = -P$ and $R^T = -R$. The equation \eqref{the condition on anticommutativity of J_alpha J_Omega in matrices} writes
    \begin{align*}
        \begin{pmatrix}
        Q^2 - PR & -QP - PQ^T \\
        RQ + Q^TR & -RP + (Q^T)^2
        \end{pmatrix} = - \epsilon_2 \epsilon_3 \mathbb{1} \ . 
    \end{align*}
The conditions $ -QP - PQ^T =  RQ + Q^TR = 0$ are satisfied automatically, since  
    \begin{align*}
        P = \begin{pmatrix}
            0 & E \\ 
            -E & 0
            \end{pmatrix} \ , && 
        Q = \begin{pmatrix} 
            B & C \\ 
            -A & -B
            \end{pmatrix} \ , && 
        R = \begin{pmatrix}
            0 & D \\ 
            -D & 0
            \end{pmatrix} \ ,
    \end{align*}
where $A,B,C,D,E$ are entries of $\underline{\alpha}$, i.e. coefficients of $\alpha$ in the canonical basis. We also notice that 
    \begin{align*}
        Q^2 - PR = - \epsilon_2 \epsilon_3 \mathbb{1} \iff -RP + (Q^T)^2 = - \epsilon_2 \epsilon_3 \mathbb{1} \ , 
    \end{align*}
which is due to antisymmetry of $A$ and $C$, and $B^2$ being diagonal. Finally
    \begin{align*}
        Q^2 - PR = 
            \begin{pmatrix}
            B^2 - AC + DE & 0 \\
            0 & B^2 - AC + DE
            \end{pmatrix} = - \epsilon_2 \epsilon_3 \mathbb{1} \ .
    \end{align*}
This concludes the proof.
\end{proof}


\begin{example}\label{example for proposition 3.1}
Let $(\Omega, \alpha)$ be a non-degenerate M-A structure. We will search for $2$-forms $\alpha$, which satisfy the condition \eqref{condition for anticommutativity of J_alpha and J_Omega} and commute with the canonical symplectic form in the sense $[\underline{\alpha}, \underline{\Omega}] = 0$, where the matrices are given by the canonical basis of $\Omega$. The commutativity condition is quite arbitrary at this point, but becomes more relevant if one is interested in construction of hyper-(para-)complex structures, or variations of generalized-Kähler structures on $T \oplus T^*$. To obtain even more specific family of solutions to \eqref{condition for anticommutativity of J_alpha and J_Omega}, we further choose $\epsilon_2 = - \epsilon_3$. Then the condition \eqref{condition for anticommutativity of J_alpha and J_Omega} becomes $\underline{\alpha}^T = \underline{\alpha}^{-1}$. Hence we search for antisymmetric elements of the orthogonal group $\operatorname{O}(4)$, which commute with $\underline{\Omega}$. If we take $Q \in \operatorname{O}(2)$, then $ \underline{\alpha} : = \begin{pmatrix} 0 & Q \\ -Q^T & 0 \end{pmatrix} $ satisfies $ \underline{\alpha}^T = - \underline{\alpha} = \underline{\alpha}^{-1}$. The commutativity with $\underline{\Omega}$ then forces $\det Q = -1$, which amounts to $Q =  \begin{pmatrix} a & b \\ b & -a \end{pmatrix}$ for $a,b \in \mathbb{R}$ such that $a^2 + b^2 = 1$. Comparing this with the general form of $\underline{\alpha}$ described in \eqref{matrices of alpha and Omega in canonical basis}, we obtain 
    \begin{align*}
        \underline{\alpha} =             
            \begin{pmatrix}
            0 & 0 & B & - A \\
            0 & 0 & -A & -B \\
            -B & A & 0 & 0 \\
            A & B & 0 & 0 
            \end{pmatrix} \ .
    \end{align*}
Therefore, the coefficients of $\alpha$ must satisfy $A = -C$, $D = E = 0$, and $\Pf (\alpha) = - A^2 - B^2 = -1$. This is in concordance with \eqref{condition for anticommutativity of J_alpha and J_Omega in canonical coordinates}. The corresponding Monge-Ampère equation, which is determined from $(\Omega, \alpha)$ by \eqref{def: M-A equation}, is
    \begin{align*}
        A f_{xx} + 2(1-A^2) f_{xy} - A f_{yy} = 0 \ ,
    \end{align*}
where $A \in C^\infty (B)$. For example, if we choose $|A| = 1$, we obtain the wave equation $f_{xx} = f_{yy}$.
\end{example}

In the above example, we have described a family of non-degenerate M-A structures (parameterized by smooth functions $A,B$ satisfying $A^2 + B^2 = 1$), which, according to proposition \ref{prop.: pair-wise anticommutative diag/antidiag structures}, determine the corresponding family of anticommutative pairs $\mathbb{J}_\alpha, \mathbb{J}_\Omega$ of GaPC and GaC structures for $(\epsilon_2, \epsilon_3) = (1, -1) $ (or GaC and GaPC structures for $(\epsilon_2, \epsilon_3) = (-1,1) $). Choosing $\epsilon_1 = -1$ extends the pair $\mathbb{J}_\alpha, \mathbb{J}_\Omega$ to a mutually anticommutative triple of generalized almost structures, with the third structure being the GaPC structure $\mathbb{J}_\rho$ (since $\Pf (\alpha) < 0$ - see proposition \ref{cor: generalized geometry from generic M-A structures}). In the following, we will show how pair-wise anticommutative triples give rise to certain quadric surfaces of generalized almost structures. 


\begin{proposition}\label{corollary: quadrics of generalized structures from M-A structure}
Let $(\Omega, \alpha) \in \Omega^2 \left( T^*\mathcal{B} \right) \times \Omega^2 \left( T^*\mathcal{B} \right)$ be a non-degenerate M-A structure, and let $\rho \in \operatorname{End} (TT^*\mathcal{B})$ be given by \eqref{def: rho_alpha tensor}. If \eqref{condition for anticommutativity of J_alpha and J_Omega} holds, then 
    \begin{align}\label{definition of bold A}
        \mathbb{A} : =
            \begin{pmatrix}
            a_1 \rho & a_2 \pi_\alpha^\# + a_3 \pi_\Omega^\# \\
            a_2 \epsilon_2 \alpha_\# +  a_3 \epsilon_3 \Omega_\#  & -a_1 \rho^*
            \end{pmatrix} \ , 
    \end{align}
where $a_i \in \mathbb{R} \ \forall i$, is a generalized almost structure, if and only if 
    \begin{align}\label{the coefficient of A squared}
         k : = - \sgn \Pf (\alpha) {a_1}^2 + {a_2}^2\epsilon_2 + {a_3}^2\epsilon_3 
    \end{align}
satisfies $|k| = 1$. 
    \begin{enumerate}
        \item If $(\Omega, \alpha)$ is elliptic, then there are the following quadrics of generalized almost structures in $\mathcal{A} : = \operatorname{span}_{\mathbb{R}} \{ \mathbb{J}_\rho, \mathbb{J}_\alpha, \mathbb{J}_\Omega \}$:
            \begin{enumerate}
                \item for $k = 1$, there are two $1$-sheeted hyperboloids, a $2$-sheeted hyperboloid and, a $2$-sphere of GaPC structures in $\mathcal{A}$,
                \item for $k = -1$, there are two $2$-sheeted hyperboloids and a $1$-sheeted hyperboloid of GaC structures in $\mathcal{A}$.
            \end{enumerate}
        \item If $(\Omega, \alpha)$ is hyperbolic, then 
            \begin{enumerate}
                \item for $k = 1$, there are two $2$-sheeted hyperboloids and a $1$-sheeted hyperboloid of GaPC structures in $\mathcal{A}$,
                \item for $k = -1$, there are two $1$-sheeted hyperoloids, a $2$-sheeted hyperboloid, and a $2$-sphere of GaC structures in $\mathcal{A}$.
            \end{enumerate} 
    \end{enumerate}
\end{proposition}


Before we move to the proof of the proposition \ref{corollary: quadrics of generalized structures from M-A structure}, we also want to emphasize that statements 1. and 2. are separate cases and should not be mixed together, for example, when searching for a generalized Kähler structure associated to $(\Omega, \alpha)$. This is because a given non-degenerate M-A structure has only one value of $\sgn \Pf (\alpha)$, distinguishing the elliptic case from the hyperbolic one. 


\begin{proof}
We have $\mathbb{A} = a_1 \mathbb{J}_\rho + a_2 \mathbb{J}_\alpha + a_3 \mathbb{J}_\Omega $, where $\mathbb{J}_\rho, \mathbb{J}_\alpha, \mathbb{J}_\Omega$ are given by \eqref{diag and antidiag structures for generic 2D M-A equation} with $\epsilon_1 = -1$. Then $\mathbb{A}^\bullet \eta = {a_1}^2 \mathbb{J}_\rho^\bullet \eta + {a_2}^2 \mathbb{J}_\alpha^\bullet \eta + {a_3}^2 \mathbb{J}_\Omega^\bullet \eta $ and 
    \begin{align}\label{intermediate computation 0}
        \begin{split}
        \mathbb{A}^2 & = {a_1}^2 {\mathbb{J}_\rho}^2 + {a_2}^2 {\mathbb{J}_\alpha}^2 + {a_3}^2 {\mathbb{J}_\Omega}^2 \\
        & \quad + a_1 a_2 \{ \mathbb{J}_\rho, \mathbb{J}_\alpha \} + a_1 a_3 \{ \mathbb{J}_\rho, \mathbb{J}_\Omega \} + a_2 a_3 \{ \mathbb{J}_\alpha, \mathbb{J}_\Omega \}  \ .
        \end{split}
    \end{align}
By proposition \ref{prop.: pair-wise anticommutative diag/antidiag structures}, our assumptions gives $\{ \mathbb{J}_\rho, \mathbb{J}_\alpha \} =  \{ \mathbb{J}_\rho, \mathbb{J}_\Omega \} = \{ \mathbb{J}_\alpha, \mathbb{J}_\Omega \} = 0 $. Thus, using relations \eqref{relations for diag. and antidiag structures}, we obtain 
    \begin{align*}
        \mathbb{A}^2 =  {a_1}^2 {\mathbb{J}_\rho}^2 + {a_2}^2 {\mathbb{J}_\alpha}^2 + {a_3}^2 {\mathbb{J}_\Omega}^2 \ , && \mathbb{A}^\bullet \eta =  {a_1}^2 \mathbb{J}_\rho^\bullet \eta - \left( {a_2}^2 \epsilon_2 + {a_3}^2 \epsilon_3 \right) \eta \ . 
    \end{align*}
By proposition \ref{lemma: diagonal antidiagonal structures} and proposition \ref{cor: generalized geometry from generic M-A structures}, the type of generalized almost structure $\mathbb{J}_\rho$ depends on the Pfaffian of the M-A structure $(\Omega, \alpha)$ as follows
    \begin{align*}
        {\mathbb{J}_\rho}^2 = 
            \begin{cases}
            \ \ \ \operatorname{Id}_{\mathbb{T}\mathcal{B}} \text{ for } \Pf(\alpha) < 0 \\
            - \operatorname{Id}_{\mathbb{T}\mathcal{B}} \text{ for } \Pf(\alpha) > 0
            \end{cases} && 
        \mathbb{J}_\rho^\bullet \eta = 
            \begin{cases}
            - \eta \text{ for } \Pf(\alpha) < 0 \\
            \ \ \eta \text{ for } \Pf(\alpha) > 0
            \end{cases}
    \end{align*}  
where $\sgn \Pf (\alpha) : = \frac{\Pf (\alpha) }{|\Pf (\alpha) |}$ (which is well-defined as we assume the M-A structure to be non-degenerate, $\Pf (\alpha) \neq 0$). Hence we arrive at
    \begin{align}\label{intermediate computation 4}
        \begin{split}
        \mathbb{A}^2 & = \left( - \sgn \Pf (\alpha) {a_1}^2 + {a_2}^2\epsilon_2 + {a_3}^2\epsilon_3 \right) \operatorname{Id}_{\mathbb{T}\mathcal{B}} \ , \\
        \mathbb{A}^\bullet \eta & = \left( -\sgn \Pf (\alpha) {a_1}^2 + {a_2}^2\epsilon_2 + {a_3}^2\epsilon_3 \right) ( - \eta ) \ ,
        \end{split}
    \end{align}
with the immediate consequence that $\mathbb{A}$ is a generalized almost structure if and only if $|k| = 1$. We proceed with the proof of statements 1. and 2. 

\hfill

\noindent \textit{1. $\Pf (\alpha) > 0$: elliptic M-A structures.} (a) For $k = 1$, we see from \eqref{intermediate computation 4} that $ \mathbb{A}^2 = \operatorname{Id}_{\mathbb{T}\mathcal{B}}$ and $\mathbb{A}^\bullet \eta = -\eta$, i.e. $\mathbb{A}$ is a GaPC structure. Moreover, positive Pfaffian implies the coefficients of $\mathbb{A}$ must satisfy $- {a_1}^2 + {a_2}^2\epsilon_2 + {a_3}^2\epsilon_3 = 1 $. This yields two $1$-sheeted hyperboloids, a $2$-sheeted hyperboloid, and a $2$-sphere, depending on the value of $\epsilon_2$ and $\epsilon_3$, which is summarized in table \ref{tab: table for Pf > 0 and k = 1}. 
    \begin{table}[!ht]
    \begin{center}
    \begin{tabular}{ | >{\centering\arraybackslash}X p{2.0cm} | >{\centering\arraybackslash}X p{4.3cm}| >{\centering\arraybackslash}X p{4.6cm}|} 
    \hline 
    $\sgn( \epsilon_2, \epsilon_3)$  & condition on $(a_1, a_2, a_3)$ & quadric of GaPC structures \\ 
    \hline
    $(+, +)$  & $-{a_1}^2 + {a_2}^2 + {a_3}^2 = 1$ & $1$-sheeted hyperboloid  \\
    \hline
    $(+, -)$  & $-{a_1}^2 + {a_2}^2 - {a_3}^2 = 1$ & $2$-sheeted hyperboloid  \\
    \hline
    $(-, +)$  & $-{a_1}^2 - {a_2}^2 + {a_3}^2 = 1$ & $2$-sheeted hyperboloid  \\
    \hline
    $(-, -)$  & $-{a_1}^2 - {a_2}^2 - {a_3}^2 = 1$ & $\emptyset$ \\
    \hline
    \end{tabular}
    \caption{\label{tab: table for Pf > 0 and k = 1} Quadrics of GaPC structures for $\Pf (\alpha) > 0$.}
    \end{center}
    \end{table}
(b) If $k = -1$, then $ \mathbb{A}^2 = -\operatorname{Id}_{\mathbb{T}\mathcal{B}}$ and $\mathbb{A}^\bullet \eta = \eta$, meaning $\mathbb{A}$ is a GaC structure, and the coefficients $a_i$ must satisfy $- {a_1}^2 + {a_2}^2\epsilon_2 + {a_3}^2\epsilon_3 = - 1$. This implies the existence of quadrics described in table \ref{tab: table for Pf > 0 and k = -1}.
    \begin{table}[!ht]
    \begin{center}
    \begin{tabular}{ | >{\centering\arraybackslash}X p{2.0cm} | >{\centering\arraybackslash}X p{4.3cm}| >{\centering\arraybackslash}X p{4.6cm}|} 
    \hline 
    $\sgn( \epsilon_2, \epsilon_3)$  & condition on $(a_1, a_2, a_3)$ & quadric of GaC structures \\ 
    \hline
    $(+, +)$  & $-{a_1}^2 + {a_2}^2 + {a_3}^2 = -1$ & $2$-sheeted hyperboloid  \\
    \hline
    $(+, -)$  & $-{a_1}^2 + {a_2}^2 - {a_3}^2 = -1$ & $1$-sheeted hyperboloid  \\
    \hline
    $(-, +)$  & $-{a_1}^2 - {a_2}^2 + {a_3}^2 = -1$ & $1$-sheeted hyperboloid  \\
    \hline
    $(-, -)$  & $-{a_1}^2 - {a_2}^2 - {a_3}^2 = -1$ & $2$-sphere \\
    \hline
    \end{tabular}
    \caption{\label{tab: table for Pf > 0 and k = -1} Quadrics of GaC structures for $\Pf (\alpha) > 0$.}
    \end{center}
    \end{table}

\hfill

\noindent \textit{2. $\Pf (\alpha) < 0$: hyperbolic M-A structures.} The situation is very similar to the elliptic case. We again obtain either (a) for $k = 1$ a GaPC structure, or, (b) for $k = -1$ a GaC structure. The main alteration is that the resulting GaPC/GaC structures differ from those obtained in the elliptic case. This is simply because for a non-degenerate M-A structure, by definition, $\sgn \Pf (\alpha)$ is constant. Thus, the distinction between hyperbolic and elliptic case happens on the level of the M-A structure, with the expected diference between the GaPC/GaC structures corresponding to different couples $(\Omega, \alpha)$. The quadrics of generalized almost structures arising from hyperbolic M-A structures are given in table \ref{tab: table for Pf < 0 and k = 1} for $k = {a_1}^2 + {a_2}^2\epsilon_2 + {a_3}^2\epsilon_3 = 1$, and in table \ref{tab: table for Pf < 0 and k = -1} for $k = {a_1}^2 + {a_2}^2\epsilon_2 + {a_3}^2\epsilon_3 = -1$. 
    \begin{table}[!ht]
    \begin{center}
    \begin{tabular}{ | >{\centering\arraybackslash}X p{2.0cm} | >{\centering\arraybackslash}X p{4.3cm}| >{\centering\arraybackslash}X p{4.6cm}|} 
    \hline 
    $\sgn( \epsilon_2, \epsilon_3)$  & condition on $(a_1, a_2, a_3)$ & quadric of GaPC structures \\ 
    \hline
    $(+, +)$  & ${a_1}^2 + {a_2}^2 + {a_3}^2 = 1$ & $2$-sphere  \\
    \hline
    $(+, -)$  & ${a_1}^2 + {a_2}^2 - {a_3}^2 = 1$ & $1$-sheeted hyperboloid  \\
    \hline
    $(-, +)$  & ${a_1}^2 - {a_2}^2 + {a_3}^2 = 1$ & $1$-sheeted hyperboloid  \\
    \hline
    $(-, -)$  & ${a_1}^2 - {a_2}^2 - {a_3}^2 = 1$ & $2$-sheeted hyperboloid \\
    \hline
    \end{tabular}
    \caption{\label{tab: table for Pf < 0 and k = 1} Quadrics of GaPC structures for $\Pf (\alpha) < 0$.}
    \end{center}
    \end{table}
    \begin{table}[!ht]
    \begin{center}
    \begin{tabular}{ | >{\centering\arraybackslash}X p{2.0cm} | >{\centering\arraybackslash}X p{4.3cm}| >{\centering\arraybackslash}X p{4.6cm}|} 
    \hline 
    $\sgn( \epsilon_2, \epsilon_3)$  & condition on $(a_1, a_2, a_3)$ & quadric of GaC structures \\ 
    \hline
    $(+, +)$  & ${a_1}^2 + {a_2}^2 + {a_3}^2 = -1$ & $\emptyset$  \\
    \hline
    $(+, -)$  & ${a_1}^2 + {a_2}^2 - {a_3}^2 = -1$ & $2$-sheeted hyperboloid  \\
    \hline
    $(-, +)$  & ${a_1}^2 - {a_2}^2 + {a_3}^2 = -1$ & $2$-sheeted hyperboloid  \\
    \hline
    $(-, -)$  & ${a_1}^2 - {a_2}^2 - {a_3}^2 = -1$ & $1$-sheeted hyperboloid \\
    \hline
    \end{tabular}
    \caption{\label{tab: table for Pf < 0 and k = -1} Quadrics of GaC structures for $\Pf (\alpha) < 0$.}
    \end{center}
    \end{table}

\end{proof}


A natural question at this point is whether any of the generalized almost structures described by proposition \ref{corollary: quadrics of generalized structures from M-A structure} coincide. In the following proposition, we will show there are no coincidences between structures parameterized by different points of either the same quadric or two different quadrics.

\begin{proposition}\label{prop.: two different points on quadrics give two different g.a. structures}
Two different points chosen arbitrarily from the quadrics described in statement 1. (if $\Pf (\alpha) > 0)$, or statement 2. (if  $\Pf (\alpha) < 0)$, of proposition \ref{corollary: quadrics of generalized structures from M-A structure} represent two different generalized almost structures. 
\end{proposition}

\begin{proof}
Let $\mathbb{A}, \mathbb{B}$ be given by \eqref{definition of bold A}. The difference between the structures may occur via the choice of coefficients $(a_i), (b_i)$, and the values of epsilons, $\epsilon_i^\mathbb{A}, \epsilon_i^\mathbb{B}$. The coefficients $(a_i), (b_i)$ determine the corresponding $k_\mathbb{A}, k_\mathbb{B}$, defined by \eqref{the coefficient of A squared}. Note that both $\mathbb{A}$ and $ \mathbb{B}$ are constructed from a single M-A structure $(\Omega, \alpha)$, so the Pfaffian is fixed and thus $\sgn \Pf (\alpha)$ is constant. This means that there will always be only one $\rho$ and only one of the statements 1. and 2. of proposition \ref{corollary: quadrics of generalized structures from M-A structure} will apply to a given non-degenerate M-A structure. 

\hfill

\noindent \textit{Case 1.} Suppose $k_\mathbb{A} = k_\mathbb{B}$ and $\epsilon_i^\mathbb{A} = \epsilon_i^\mathbb{B}$ for $i = 2,3$. Then $\mathbb{A}$ and $\mathbb{B}$ are $\mathbb{R}$-linear combinations of the same generalized almost structures, i.e. $\mathbb{A} = a_1 \mathbb{J}_\rho + a_2 \mathbb{J}_\alpha + a_3 \mathbb{J}_\Omega$ and $\mathbb{B} = b_1 \mathbb{J}_\rho + b_2 \mathbb{J}_\alpha + b_3 \mathbb{J}_\Omega$. This means 
    \begin{align*}
        \mathbb{A} = \mathbb{B} \quad \iff \quad a_i = b_i \ \forall i \ .
    \end{align*}
Note that this case amounts to choosing the structures $\mathbb{A}, \mathbb{B}$ on the same quadric. 

\hfill

\noindent \textit{Case 2.} Suppose $k_\mathbb{A} = k_\mathbb{B}$ but $\epsilon_i^\mathbb{A} \neq \epsilon_i^\mathbb{B}$ for at least one $i$. Since the structure $\mathbb{J}_\rho$ is fixed for both the $\mathbb{A}$ and $\mathbb{B}$ (due to the same $\Pf (\alpha)$), we have $\mathbb{A} = \mathbb{B}$, if and only if $a_1 = b_1$ and 
    \begin{align*}
        \begin{pmatrix}
        0 & a_2 \pi_\alpha^\# + a_3 \pi_\Omega^\#  \\ 
        a_2 \epsilon_2^\mathbb{A} \alpha_\# +  a_3 \epsilon_3^\mathbb{A} \Omega_\# & 0
        \end{pmatrix} =
        \begin{pmatrix}
        0 & b_2 \pi_\alpha^\# + b_3 \pi_\Omega^\#  \\ 
        b_2 \epsilon_2^\mathbb{B} \alpha_\# +  b_3 \epsilon_3^\mathbb{B} \Omega_\# & 0
        \end{pmatrix} \ .
    \end{align*}
Comparing the upper right blocks we get 
    \begin{align}\label{intermediate computation 3}
        (a_2 - b_2) \pi_\alpha^\# = (b_3 - a_3) \pi_\Omega^\# \ .
    \end{align}
By definition of the M-A structure $(\Omega, \alpha)$, we have $\alpha \wedge \Omega = 0$. This implies that the dual bivectors satisfy $\pi_\alpha \wedge \pi_\Omega = 0$. Consequently, $\pi_\alpha^\#$ cannot be written as a non-zero multiple of $\pi_\Omega^\#$, and thus \eqref{intermediate computation 3} holds, if and only if $a_2 = b_2$ and $a_3 = b_3$.

\hfill

\noindent \textit{Case 3.}  Suppose $k_\mathbb{A} \neq k_\mathbb{B}$. Then $\mathbb{A}$ is a generalized almost structure of a different type than $\mathbb{B}$, which concludes the proof.
\end{proof}


Propositions \ref{corollary: quadrics of generalized structures from M-A structure} and \ref{prop.: two different points on quadrics give two different g.a. structures} give the following theorem. 

\begin{theorem}\label{thm: generalized geometry and M-A structures}
Let $(\Omega, \alpha) \in \Omega^2 \left( T^*\mathcal{B} \right) \times \Omega^2 \left( T^*\mathcal{B} \right)$ be a non-degenerate M-A structure on $T^*\mathcal{B}$, and let $\rho \in \operatorname{End} (TT^*\mathcal{B})$ be given by \eqref{def: rho_alpha tensor}. Then $(\Omega, \alpha)$ defines by \eqref{definition of bold A} a family of generalized almost geometries $\mathbb{A} \in \operatorname{End} \left(\mathbb{T} T^*\mathcal{B} \right)$ on $T^*\mathcal{B}$. They are parameterized by quadric surfaces in $\mathbb{R}^3$, and described in tables \ref{tab: table for Pf > 0 and k = 1} - \ref{tab: table for Pf < 0 and k = -1}. Two different points of any of the quadrics parameterize two different geometries.
\end{theorem}


\noindent \textbf{On the necessity of anticommutativity assumptions.} Considering the assumptions of proposition \ref{corollary: quadrics of generalized structures from M-A structure}, we now focus on the question whether the pair-wise anticommutativity (which was assumed in the corollary by the choice $\epsilon_1 = -1$ in \eqref{definition of bold A} and by the equation $\epsilon_3 \pi_\alpha^\# \Omega_\# + \epsilon_2 \pi_\Omega^\# \alpha_\# = 0$) is a~necessary condition for $\mathbb{A} = a_1 \mathbb{J}_\rho + a_2 \mathbb{J}_\alpha + a_3 \mathbb{J}_\Omega$ to be a generalized almost structure. We want to consider the situation where all three generalized almost structures $\mathbb{J}_\rho, \mathbb{J}_\alpha, \mathbb{J}_\Omega$ given by \eqref{diag and antidiag structures for generic 2D M-A equation} are contributing to $\mathbb{A}$. So we assume $a_i \neq 0$ for all $i$. Without any assumption about the anticommutativity, $\mathbb{A}^2$ is described by \eqref{intermediate computation 0}. If $\mathbb{A}$ is a generalized almost structure, then $\mathbb{A}^2 = \pm \operatorname{Id}_{\mathbb{T}\mathcal{B}}$ and by proposition \ref{cor: generalized geometry from generic M-A structures} 
    \begin{align}\label{sum of squares of g.a.s.}
        {a_1}^2 {\mathbb{J}_\rho}^2 + {a_2}^2 {\mathbb{J}_\alpha}^2 + {a_3}^2 {\mathbb{J}_\Omega}^2 = c_1 \operatorname{Id}_{\mathbb{T}\mathcal{B}} \ , 
    \end{align}
for appropriate $c_1 \in \mathbb{R}$. This means 
    \begin{align*}
    a_1 a_2 \{ \mathbb{J}_\rho, \mathbb{J}_\alpha \} + a_1 a_3 \{ \mathbb{J}_\rho, \mathbb{J}_\Omega \} + a_2 a_3 \{ \mathbb{J}_\alpha, \mathbb{J}_\Omega \} = c_2 \operatorname{Id}_{\mathbb{T}\mathcal{B}} \ ,     
    \end{align*}
for some $c_2 \in \mathbb{R}$. Since both products $\mathbb{J}_\rho \mathbb{J}_\alpha$ and $\mathbb{J}_\rho \mathbb{J}_\Omega$ have only zeros in the diagonal blocks, the last equation implies
    \begin{align*}
          \{ \mathbb{J}_\rho, \mathbb{J}_\alpha \} = 0 \ , &&  \{ \mathbb{J}_\rho, \mathbb{J}_\Omega \} = 0 \ , && a_2 a_3 \{ \mathbb{J}_\alpha,  \mathbb{J}_\Omega \} = c_2 \operatorname{Id}_{\mathbb{T}\mathcal{B}} \ .
    \end{align*}
By proposition \ref{prop.: pair-wise anticommutative diag/antidiag structures}, the first two equations are satisfied, if and only if $\epsilon_1 = -1$ in $\mathbb{J}_\rho = \begin{pmatrix} \rho & 0 \\ 0 & \epsilon_1 \rho^* \end{pmatrix}$. Moreover, if $\epsilon_1 = -1$, then from proposition \ref{lemma: diagonal antidiagonal structures} follows that \eqref{sum of squares of g.a.s.} is equivalent to 
    \begin{align*}
       c_1 = - \sgn \Pf (\alpha) {a_1}^2 + {a_2}^2\epsilon_2 + {a_3}^2\epsilon_3 \ . 
    \end{align*}
Note that $c_1$ coincides with $k$ described in \eqref{the coefficient of A squared}. Now let us assume that the equation $a_2 a_3 \{ \mathbb{J}_\alpha,  \mathbb{J}_\Omega \} = c_2 \operatorname{Id}_{\mathbb{T}M}$ can be solved. As a consequence of $\mathbb{A}^2 = \pm \operatorname{Id}_{\mathbb{T}\mathcal{B}}$, the constants $c_1, c_2$ must satisfy $|c_1 + c_2| = 1$. At the same time, proposition \ref{lemma: diagonal antidiagonal structures} implies $\mathbb{A}^\bullet \eta = -c_1 \eta$, and we need $|c_1| = 1$. 

Comming back to the equation $a_2 a_3 \{\mathbb{J}_\alpha,  \mathbb{J}_\Omega \} = c_2 \operatorname{Id}_{\mathbb{T}\mathcal{B}}$. In matrices, this corresponds to  
    \begin{align}\label{intermediate system 1}
        \begin{split}
        \epsilon_2 \epsilon_3 \mathbb{1} + (\underline{\alpha} \underline{\Omega})^2 & = \frac{-c_2 \epsilon_2}{a_2 a_3} \underline{\alpha} \underline{\Omega} , \\
        \epsilon_2 \epsilon_3 \mathbb{1} + (\underline{\Omega} \underline{\alpha})^2 & = \frac{-c_2 \epsilon_2}{a_2 a_3} \underline{\Omega} \underline{\alpha} \ .
        \end{split}
    \end{align}
Since $ (\underline{\alpha} \underline{\Omega})^2 = (\underline{\Omega} \underline{\alpha})^2$, the system \eqref{intermediate system 1} can be solved, if and only if $[\underline{\alpha}, \underline{\Omega}] = 0$. In this case, the two equations coincide. Note that the requirement $[\underline{\alpha}, \underline{\Omega}] = 0$ is not present if we want $\{\mathbb{J}_\alpha, \mathbb{J}_\Omega\} = 0$ (which corresponds to $c_2 = 0$). This is because the right-hand side of \eqref{intermediate system 1} vanish if we rewrite $\{\mathbb{J}_\alpha,  \mathbb{J}_\Omega \} = 0$ as a~matrix equation. 

Thus we arrived at the following conclusion. If we do not assume a priori any relation between $\mathbb{J}_\rho, \mathbb{J}_\alpha, \mathbb{J}_\Omega$, then for $\mathbb{A}$ to be a generalized almost structure, $\mathbb{J}_\rho$ must anticommute with $\mathbb{J}_\alpha$ and $\mathbb{J}_\Omega$. The assumption $ \{ \mathbb{J}_\alpha,  \mathbb{J}_\Omega \} = 0$ can be replaced with $ \{ \mathbb{J}_\alpha,  \mathbb{J}_\Omega \} = c \operatorname{Id}_{\mathbb{T}\mathcal{B}}$, with the price given by the necessity of $[\underline{\alpha}, \underline{\Omega}] = 0$, as well as the requirement $|c + k| = 1$, which is in addition to $|k| = 1$, where $k$ is given by \eqref{the coefficient of A squared}.


\hfill 

\noindent \textbf{On the link between M-A structures and M-A equations.} We have seen in propositions \ref{cor: generalized geometry from generic M-A structures} and \ref{corollary: quadrics of generalized structures from M-A structure} how a non-degenerate M-A structure gives rise to various generalized almost structures. We have also described the link between M-A structures and M-A equations, which is provided by \eqref{def: M-A equation}. Now we want to use this link to define generalized geometries associated with a given M-A equation. To do this, we have to take care of ambiguities (and the corresponding well-definedness issue) associated with the following observation. 

Consider non-degenerate M-A structures $(\Omega, \alpha)$ and $(\Omega,  \tilde{\alpha})$. By definition, $\alpha \wedge \Omega = 0 = \Tilde{\alpha} \wedge \Omega$. Moreover, for $h \in C^\infty (\mathcal{B})$
    \begin{align*}
        (\D f )^* (h \alpha) = h|_{\operatorname{Im} \D f} (\D f)^*\alpha \ ,
    \end{align*}
where $f \in C^\infty (\mathcal{B})$ and $\operatorname{Im} \D f$ is understood in the sense $\D f \colon B \to T^*\mathcal{B}, x \mapsto \D_x f$. Consequently, if $h$ is everywhere non-zero,
    \begin{align*}
        (\D f)^* \alpha = 0 \quad  \iff \quad  (\D f)^* (h \alpha) = 0 \ .
    \end{align*}
Thus two non-degenerate M-A structures give rise to the same M-A equation, if and only if $\Tilde{\alpha} = h \alpha$. Denote by $[\alpha]$ the equivalence class of $2$-forms satisfying $\alpha \wedge \Omega = 0$, where 
    \begin{align*}
        \Tilde{\alpha} \in [\alpha] \quad \overset{\text{def.}}{\iff} \quad \Tilde{\alpha} = h \alpha \ ,
    \end{align*}
for a non-vanishing $h \in C^\infty (T^*\mathcal{B}$. Then there is the following $1-1$ correspondence between equivalence classes and M-A equations
    \begin{align*}
        \{ [\alpha] \ | \ \alpha \in \Gamma \left(\Lambda^2 T^*\mathcal{B} \right) \} \longleftrightarrow \{ (\D f)^* \alpha = 0 \ | \  f \in C^\infty(\mathcal{B}) \} \ .
    \end{align*}
To put this in the context of generalized geometries associated with M-A structures, we want to see how the change of representative in the equivalence class $[\alpha]$ changes the family of generalized almost structures given by proposition \eqref{cor: generalized geometry from generic M-A structures}. Recall that $\mathbb{A}$ determined by a non-degenerate M-A structure $(\Omega, \alpha)$ is defined via three generalized almost structures given by \eqref{diag and antidiag structures for generic 2D M-A equation} with $\epsilon_1 = -1$, and satisfying the condition $\epsilon_3 \pi_\alpha^\# \Omega_\# + \epsilon_2 \pi_\Omega^\# \alpha_\# = 0$. 

Consider a non-degenerate M-A structure $(\Omega, \Tilde{\alpha})$, where $\Tilde{\alpha} \in [\alpha]$. Firstly, it is clear that $\mathbb{J}_\Omega$ is the same for both M-A structures, since $\Omega$ remains the same. Secondly, we compute the change in $\rho$ and the corresponding $\mathbb{J}_\rho$. Starting with the Pfaffian, 
    \begin{align}\label{On the link between M-A structures and M-A equations: Pfaffian}
         \tilde{\alpha} \wedge \tilde{\alpha} = h^2 \Pf (\alpha) \Omega \wedge \Omega \quad \Rightarrow \quad \Pf (\tilde{\alpha}) =  h^2 \Pf (\alpha) \ .
    \end{align}
Now using the definition \eqref{def: rho_alpha tensor}, we have
    \begin{align}\label{On the link between M-A structures and M-A equations: rho}
         \tilde{\rho} = \frac{1}{ \sqrt{ | \Pf (\tilde{\alpha}) | } } \pi_\Omega^\# \Tilde{\alpha}_\# = \frac{h}{ |h| \sqrt{ | \Pf ( \alpha) | } }  \pi_\Omega^\# \alpha_\# = \sgn h \rho \ , 
    \end{align}
with the immediate consequence
     \begin{align}\label{eq: change of J given by changing representative of [alpha]}
         \mathbb{J}_{ \tilde{\rho} } = \sgn h \mathbb{J}_\rho \ .
     \end{align}
Thirdly, we focus on $\mathbb{J}_\alpha$. We have $ \Tilde{\alpha}_\# = h \alpha_\#$ and $ \pi_{ \Tilde{\alpha} }^\# = \frac{1}{h} \pi_\alpha^\#$, which follows from $C^\infty$-linearity of tensor fields and the definition of $\pi_\alpha^\#$. Thus $\alpha \mapsto \Tilde{\alpha} = h \alpha$ results in 
    \begin{align}\label{On the link between M-A structures and M-A equations: J_alpha}
        \mathbb{J}_\alpha = 
        \begin{pmatrix}
        0  & \pi_\alpha^\# \\
        \epsilon_2 \alpha^\# & 0
        \end{pmatrix} \quad \longmapsto \quad 
        \begin{pmatrix}
        0 & \frac{1}{h} \pi_\alpha^\# \\
        h \epsilon_2 \alpha^\# & 0
        \end{pmatrix} = \mathbb{J}_{ \tilde{\alpha} } 
    \end{align}
Now we are ready to investigate the anticommutativity of generalized almost strutures associated with $(\Omega, \tilde{\alpha})$. We have $\{\mathbb{J}_{ \tilde{\rho} } , \mathbb{J}_\Omega \} = \sgn h \{\mathbb{J}_\rho , \mathbb{J}_\Omega \} $, and thus 
    \begin{align*}
        \{\mathbb{J}_{ \tilde{\rho} } , \mathbb{J}_\Omega \} = 0 \quad \iff \quad \epsilon_1 = -1 \ .
    \end{align*}
Proceeding with the anticommutator of $\mathbb{J}_{ \tilde{\rho} }$ and $\mathbb{J}_{ \tilde{\alpha} }$, we have 
    \begin{align*}
        \{ \mathbb{J}_{ \tilde{\rho} } ,  \mathbb{J}_{ \tilde{\alpha} } \} = \sgn h
            \begin{pmatrix}
            0 & \frac{1}{h} \left( \rho \pi_\alpha^\# + \epsilon_1 \pi_\alpha^\# \rho^* \right)  \\
            h \left( \epsilon_1 \epsilon_2 \rho^* \alpha_\# + \epsilon_2 \alpha_\# \rho \right) &  0 
            \end{pmatrix} \ , 
    \end{align*}
which implies 
    \begin{align*}
        \{ \mathbb{J}_{ \tilde{\rho} } ,  \mathbb{J}_{ \tilde{\alpha} } \} = 0 \quad \iff \quad 
            \begin{cases}
            0 = \rho \pi_\alpha^\# + \epsilon_1 \pi_\alpha^\# \rho^*  \\
            0 = \epsilon_1 \epsilon_2 \rho^* \alpha_\# + \epsilon_2 \alpha_\# \rho  
            \end{cases} 
    \end{align*}
Considering the right side of the equivalence, we have seen in the proof of proposition \ref{prop.: pair-wise anticommutative diag/antidiag structures} that the two equations are equivalent, and are satisfied, if and only if $\epsilon_1 = -1$. Finally, we consider the anticommutator of  $ \mathbb{J}_{ \tilde{\alpha} } $ and $ \mathbb{J}_\Omega $, after which we discuss the overall dependence of $\mathbb{A}$ on the choice of the representative in the class $[\alpha]$. The anticommutator is 
    \begin{align*}
        \{ \mathbb{J}_{ \tilde{\alpha} } , \mathbb{J}_\Omega \} =
            \begin{pmatrix}
            h \epsilon_3 \pi_\alpha^\# \Omega_\# + \frac{\epsilon_2}{h} \pi_\Omega^\# \alpha_\# & 0 \\
            0 & h \epsilon_2 \alpha_\# \pi_\Omega^\# + \frac{\epsilon_3}{h} \Omega_\# \pi_\alpha^\#
            \end{pmatrix} \ .
    \end{align*}
Analogously as with the previous anticommutators, vanishing of $ \{ \mathbb{J}_{ \tilde{\alpha} } , \mathbb{J}_\Omega \}$ boils down to one equation
    \begin{align*}
        \{ \mathbb{J}_{ \tilde{\alpha} } , \mathbb{J}_\Omega \} = 0 \quad \iff \quad h^2 \epsilon_3 \pi_\alpha^\# \Omega_\# + \epsilon_2 \pi_\Omega^\# \alpha_\# = 0 \ . 
    \end{align*}
In canonical coordinates, this amounts to
    \begin{align}\label{the condition on anticommutativity of J tilde alpha J Omega in matrices}
        (\underline{\alpha} \underline{\Omega})^2 = \frac{- \epsilon_2 \epsilon_3}{h^2} \mathbb{1} \ . 
    \end{align}
Now we are ready to compare the situation between $(\Omega, \alpha)$ and $(\Omega, \tilde{\alpha} )$.

From the above discussion we see that there are exactly two obstructions for $(\Omega, \alpha)$ and $(\Omega, \tilde{\alpha} )$ to determine the same family of generalized almost structures. The first obstruction follows from the change of $\mathbb{J}_\alpha$, described in \eqref{On the link between M-A structures and M-A equations: J_alpha}. The second obstruction, which is a consequence of the first one on the level of anticommutators is contained in the equation \eqref{the condition on anticommutativity of J tilde alpha J Omega in matrices}. Both obstructions can be avoided by choosing $h = 1$, but this corresponds to leaving the representative of $[\alpha]$ unchanged. 


\begin{proposition}\label{prop: two representatives of [alpha] yields different g.a.s.}
Two non-degenerate M-A structures $(\Omega, \alpha)$ and $(\Omega, \tilde{\alpha} )$, where $\tilde{\alpha} \in [\alpha]$, give rise to the same family of generalized almost structures determined by proposition \ref{corollary: quadrics of generalized structures from M-A structure}, if and only if $\Tilde{\alpha} = \alpha$. 
\end{proposition}


The anticommutativity between $\mathbb{J}_{ \tilde{\rho} }$ and $\mathbb{J}_\Omega$ (as well as $\mathbb{J}_{ \tilde{\alpha} }$) is satisfied by the same condition for all representatives of $[\alpha]$, namely by choosing $\epsilon_1 = -1$. Since the change of representative of $[\alpha]$ transforms $\mathbb{J}_\rho$ according to \eqref{eq: change of J given by changing representative of [alpha]}, we arrive at the following

\begin{proposition}\label{prop.: 1-1 correspondence between M-A equations and g.a. structures}
There is a one-to-one correspondence between 2D symplectic non-degenerate M-A equations and equivalence classes $[\mathbb{A}]$ of generalized almost structures given by proposition \ref{corollary: quadrics of generalized structures from M-A structure} with $a_2 = 0$. Representatives of $[\mathbb{A}]$ have the same type. 
\end{proposition}

\begin{proof}
Based on the discussion preceding the proposition, the only thing to notice is that $a_2 = 0$ means that $\mathbb{A}$ is constructed only from $\mathbb{J}_\rho$ and $\mathbb{J}_\Omega$. This further implies that the condition \eqref{condition for anticommutativity of J_alpha and J_Omega}, which is equivalent to anticommutativity between $\mathbb{J}_\alpha$ and $\mathbb{J}_\Omega$, does not have to be satisfied due to \eqref{intermediate computation 0}. Of course, $a_2 = 0$ also affects the quadrics described in statements 1. and 2. of the proposition. Some of them will cease to exist (such as the $2$-sheeted hyperboloid of GaPC structures corresponding to $\operatorname{Pf} (\alpha) > 0$ and $(\epsilon_2, \epsilon_3) = (1, -1)$) and those that survive will become quadric curves instead of surfaces. Finally, $\tilde{\mathbb{A}} \in [\mathbb{A}] \iff \tilde{a}_1 = \sgn h a_1$ implies together with \eqref{On the link between M-A structures and M-A equations: Pfaffian}, that the two representatives of $[\mathbb{A}]$ have the same type.
\end{proof}

The notion of normalization \eqref{def: normalization} yields a unified way of choosing a representative in $[\alpha]$. This leads to the following definition and the subsequent theorem.


\begin{definition}
Let $(\D f)^* \alpha = 0$ be a 2D symplectic Monge-Ampère equation satisfying $Pf (\alpha) \neq 0$. \textit{A generalized almost geometry associated with the M-A equation} is a generalized almost geometry determined by the normalized M-A structure $(\Omega, n(\alpha) )$ corresponding to the equation. 
\end{definition}


Now we are ready to summarize by the following theorem, which is based on propositions \ref{prop: two representatives of [alpha] yields different g.a.s.} and \ref{prop.: 1-1 correspondence between M-A equations and g.a. structures}. 

\begin{theorem}
Let $(\D f)^* \alpha = 0$ be a 2D symplectic Monge-Ampère equation on $\mathcal{B}$, such that $Pf (\alpha) \neq 0$ everywhere on $T^* \mathcal{B}$. Then there is a unique family of generalized almost geometries $\mathbb{A} \in \operatorname{End} \left(\mathbb{T} T^*\mathcal{B} \right)$ on $T^*\mathcal{B}$, determined by theorem \ref{thm: generalized geometry and M-A structures}, which are associated with the equation. 
\end{theorem}


\section*{Conlusion and Outlooks}

We have described a construction of generalized almost geometries determined by non-degenerate 2D symplectic Monge-Ampère structures and the corresponding PDEs. Inspired by the results in \cite{BANOS2007841, Crainic2004GeneralizedCS, Gualtierri2011,hu2019commuting, Kosmann-Schwarzbach2010, Salvai2015, VAISMAN201584}, we constructed many new generalized almost geometries derived from Monge-Ampère structures and from geometric objects they define. We have shown that non-degenerate Monge-Ampère structures give rise to quadric surfaces of generalized almost geometries. We also discussed the link between Monge-Ampère structures and Monge-Ampère equations in this context. In future work, we will be interested in constructing generalized geometries in dimensions higher than two, particularly in dimension three, where the situation is much richer. We will also focus on studying various notions of integrability, especially the notion of weak integrability in our context, as well as the closely related notion of Bismut connections \cite{hu2019commuting, Andrada2022, VAISMAN201584}. Particularly interesting would be to find further links between (weak) integrability of generalized structures and the local equivalence problem of certain Monge-Apmère PDEs in dimension three.

\bibliographystyle{plain}

\end{document}